\documentclass[final,leqno,onefignum,onetabnum]{siamltex1213}

\usepackage{amsmath}
\usepackage{amssymb}
\usepackage{graphicx}
\usepackage{epsfig}

\newtheorem{remark}[theorem]{Remark}

\def\l({\left(}
\def\r){\right)}
\def\mc{\mathcal}
\def\T{\mathcal{T}}
\def\A{\mathcal{A}}
\def\N{\mathbb{N}}
\def\R{\mathbb{R}}
\def\H{\hat{H}}
\def\E{\mathcal{E}}
\def\eps{\varepsilon}
\def\argmin{\operatorname{argmin}}

\title{On a $\Gamma$-convergence analysis of a quasicontinuum method\thanks{This work is supported by a grant of the Deutsche Forschungsgemeinschaft (DFG) SCHL 1706/2-1.}} 

\author{Mathias Sch{\"a}ffner and Anja Schl\"omerkemper\thanks{University of W\"urzburg, Department of Mathematics, Emil-Fischer-Stra\ss{}e 40, 97074 W\"urzburg, Germany.  (\email{mathias.schaeffner@mathematik.uni-wuerzburg.de}, \email{anja.schloemerkemper@mathematik.uni-wuerzburg.de})}}

\begin{document}
\maketitle
\slugger{mms}{xxxx}{xx}{x}{x--x}%slugger should be set to mms, siap, sicomp, sicon, sidma, sima, simax, sinum, siopt, sisc, or sirev

\begin{abstract}
In this article, we investigate a quasicontinuum method by means of analytical tools. More  precisely, we compare a discrete-to-continuum analysis of an atomistic one-dimensional model problem with a corresponding quasicontinuum model. We consider next and next-to-nearest neighbour interactions of Lennard-Jones type and focus on the so-called quasinonlocal quasicontinuum approximation. Our analysis, which applies $\Gamma$-convergence techniques, shows that, in an elastic setting, minimizers and the minimal energies of the fully atomistic problem and its related quasicontinuum approximation have the same limiting behaviour as the number of atoms tends to infinity. In case of fracture this is in general not true. It turns out that the choice of representative atoms in the quasicontinuum approximation has an impact on the fracture energy and on the location of fracture. We give sufficient conditions for the choice of representative atoms such that, also in case of fracture, the minimal energies of the fully atomistic energy and its quasicontinuum approximation coincide in the limit and such that the crack is located in the atomistic region of the quasicontinuum model as desired. 
\end{abstract}

\begin{keywords}Quasicontinuum method; atomistic-to-continuum; $\Gamma$--convergence; fracture.\end{keywords}

\begin{AMS}49J45, 74R10, 74G15, 74G10, 74G65, 70C20\end{AMS}

\pagestyle{myheadings}
\thispagestyle{plain}
\markboth{Mathias Sch{\"a}ffner and Anja Schl\"omerkemper}{On a $\Gamma$-convergence analysis of a quasicontinuum method}

\section{Introduction}

The quasicontinuum (QC) method was introduced by Tadmor, Ortiz and Phillips \cite{TOP} as a computational tool for atomistic simulations of crystalline solids at zero temperature. The key idea is to split the computational domain into regions where a very detailed (atomistic, nonlocal) description is needed and regions where a coarser (continuum, local) description is sufficient. The QC-method and improvements of it are successfully used to study crystal defects such as dislocations, nanoindentations or cracks and their impact on the overall behaviour of the material, see e.g.\ \cite{TM02}. 

There are various types of QC-methods: Some are formulated in an energy based framework, some in a force based framework; further, different couplings between the atomistic and continuum parts and different models in the continuum region are considered. In the previous decade, many articles related to the numerical analysis of such coupling methods were published. We refer to \cite{E, LO} for recent overviews, in particular on the large literature including work on error analysis. 

In this article, we consider a one-dimensional problem and focus on the so-called quasinonlocal quasicontinuum (QNL) method, first proposed in \cite{SMSJ}. The QNL-method and further generalizations of it (see e.g.~\cite{ELY,shap}) are energy-based QC-methods and are constructed to overcome asymmetries (so called ghost-forces) at the atomistic/continuum interface which arise in the classical energy based QC-method.\\
We are interested in an analytical approach in order to verify the QNL-method as an appropriate mechanical model by means of a discrete-to-continuum limit. This is embedded into the general aim of deriving continuum theories from atomistic models, see e.g.\ \cite[Section 4.1]{Ball}, where also the need of a rigorous justification of QC-methods is addressed.\\
Our approach, announced in \cite{SchaeSchl}, is based on $\Gamma$-convergence, which is a notion for the convergence of variational problems, see e.g.\ \cite{Bbeg}. We start with a one-dimensional fully atomistic model problem which takes nearest and next-to-nearest neighbour interactions into account. The limiting behaviour of the corresponding discrete model was analyzed by means of $\Gamma$-convergence techniques in \cite{NNN} for a large number of atoms. In particular the $\Gamma$-limit and the first order $\Gamma$-limit are derived there, which take into account boundary layer effects.  \\ 
From the fully atomistic model problem, we construct an approximation based on the QNL-method. In particular, we keep the nearest and next-to-nearest neighbour interactions in the atomistic (nonlocal) region and approximate the next-to-nearest neighbour interactions in the continuum (local) region by certain nearest neighbour interactions as outlined below. Furthermore, we reduce the degree of freedom of the energy by fixing certain representative atoms and let the deformation of all atoms depend only on the deformation of these representative atoms. \\
It turns out that the choice of the representative atoms has a considerable impact on the validity of the QC-method, see Theorem~\ref{5:cor}, which is the main result of this work. This theorem asserts that the QC-method is valid if the representative atoms are chosen in such a way that there is at least one non-representative atom between two neighbouring representative atoms in the local region and in particular at the interface between the local and nonlocal regions. In Proposition~\ref{lemma:fracture}, we prove that the mentioned sufficient condition on the choice of the representative atoms is indeed sharp by showing that in cases where the condition is not satisfied the limiting energy functional of the QC-method does not have the same minima as the limiting energy of the fully atomistic model and thus should not be considered an appropriate approximation.   This implies by means of analytical tools that in numerical simulations of fracture one has to make sure to pick a sufficiently large mesh in the continuum region and at the interface.

The outline of this article is as follows. In Section~2 we present the two discrete models, namely the fully atomistic and the quasicontinuum model, in detail.
In Sections~3 and 4 we investigate the limiting behaviour of the quasicontinuum energy functional by deriving the $\Gamma$-limits of zeroth and first order. It turns out that the $\Gamma$-limit of zeroth order of the fully atomistic and the quasicontinuum model coincide (Theorem~\ref{theorem:zero}). If the boundary conditions are such that the specimen behaves elastically, we prove that both models also have the same $\Gamma$-limit of first order (Theorem~\ref{theorem:elastic}). \\
If the boundary conditions are such that fracture occurs, the quasicontinuum approximation leads to a $\Gamma$-limit of first order (Theorem~\ref{theorem:fracture}) that is in general different from the one obtained earlier for the fully atomistic model (\cite{NNN}, cf.\ Theorem~\ref{th:fracturennn}). To compare the fully atomistic and the quasicontinuum model also in this regime, we analyze the $\Gamma$-limits of first order further in Section~5. As mentioned above, it turns out that if we use a sufficiently coarse mesh in the continuum region, the minimal energies of the two first order $\Gamma$-limits coincide (Theorem~\ref{5:cor}). In fact we are able to show that in our particular model problem it is sufficient that the mesh size in the continuum region is at least twice the atomistic lattice distance. With this choice, fracture occurs always in the atomistic region as desired.\\
Furthermore, the $\Gamma$-convergence results imply, under suitable assumptions, a rate of convergence of the minimal energy of the quasicontinuum model to the minimal energy of the fully atomistic model (Theorem \ref{th:limmin}). Finally, we show that the condition on the mesh size is sharp. In Proposition~\ref{lemma:fracture}, we provide examples where the corresponding $\Gamma$-limit has a different minimal energy and different minimizers than the fully atomistic system, which is due to poorly chosen meshes. This yields an analytical understanding of why meshes have to be chosen coarse enough in the continuum region.

Similar models as the one we consider here, were investigated previously in terms of numerical analysis. We refer especially to \cite{DL,LiLu,MY09,O11,OW11} where the QNL method is studied in one dimension. By proving notions of consistency and stability, those authors perform an error analysis in terms of the lattice spacing. To our knowledge, most of the results do not hold for ``fractured'' deformations.
However, in \cite{OS} a Galerkin approximation of a discrete system is considered and error bounds are proven also for states with a single crack of which the position is prescribed. Recently, a different approach based on bifurcation theory is used in \cite{LM14} to study the QC-approximation in the context of crack growth. 
%\ar

In \cite{BLBL}, a different one-dimensional atomistic-continuum coupling method is investigated. Similar as in the QC-method the domain is splitted in a discrete and a continuum region. In the discrete part the energy is given by nearest neighbour Lennard-Jones interaction and in the continuum part by an integral functional with Lennard-Jones energy density. It is shown that fracture is more favourable in the continuum than in the discrete region. To overcome this, the energy density of the continuum model is modified by introducing an additional term which depends on the lattice distance in the discrete region. Furthermore, in \cite[p.~420]{BLBLi} it is remarked that if the continuum model is replaced by a typical discretized version, the fracture is favourable in the discrete region. As mentioned above, we here treat a similar issue in the QNL-method, see in particular Theorem~\ref{5:cor}, Proposition~\ref{lemma:fracture}.

The techniques of our analysis of the QNL method are related to earlier approaches based on $\Gamma$-convergence for the passage from discrete to continuum models in one dimension, see \cite{BCi,BDMG,BG,BGhom,BLO,NNN,SSZ12}; see also \cite{FS14,FS14b} for a treatment of two dimensional models. Recently, $\Gamma$-convergence was used in \cite{ECKO} to study a QC approximation. In \cite{ECKO} a different atomistic model, namely a harmonic and defect-free crystal, is considered. Under general conditions it is shown that a quasicontinuum approximation based on summation rules has the same continuum limit as the fully atomistic system. 

Common in all those works based on $\Gamma$-convergence is that primarily information about the global minimum and minimizers are obtained.  Since atomistic solutions are not necessary global minimizers, it would be of interest to obtain also results for local minimizers, for instance in the lines of \cite{Bloc,BDMG}.  In this article, we treat systems with nearest and next-to-nearest neighbour interaction. A natural question is how the sufficient conditions on the choice of representative atoms change if we consider also $k$ interacting neighbours, $k>2$. Therefore the corresponding fully atomistic model has first to be studied, which is part of ongoing research.

\section{Setting of the Problem}

First we describe our atomistic model problem which is the same as in \cite{NNN}. We consider a one-dimensional lattice given by $\lambda_n\mathbb{Z}\cap[0,1]$ with $\lambda_n=\frac{1}{n}$ and interpret this as a chain of $n+1$ atoms. We denote by $u:\lambda_n\mathbb{Z}\cap[0,1]\to\mathbb{R}$ the deformation of the atoms from the reference configuration and write $u(i\lambda_n)=u^i$ as shorthand. We identify such functions with their piecewise affine interpolations and define
\[\A_n(0,1):=\left\{u\in C([0,1]):u\text{ is affine on }(i,i+1)\lambda_n,~i\in\{0,...,n-1\}\right\}.\]
The energy of a deformation $u\in\A_n(0,1)$ is given by
\begin{equation*}
 H_n(u)=\sum_{i=0}^{n-1}\lambda_n J_1\left(\frac{u^{i+1}-u^i}{\lambda_n}\right)+\sum_{i=0}^{n-2}\lambda_n J_2\left(\frac{u^{i+2}-u^i}{2\lambda_n}\right),\label{eq:fullatomenergy}
\end{equation*}
where $J_1$ and $J_2$ are potentials of Lennard-Jones type which will be specified in [LJ1]--[LJ4] below. Moreover, we impose boundary conditions on the first and last two atoms. For given $\ell, u_0^{(1)},u_1^{(1)}>0$ we set
\begin{equation}
 u^0=0,\quad u^1=\lambda_nu_0^{(1)},\quad u^{n-1}=\ell-\lambda_nu_1^{(1)},\quad u^n=\ell.\label{def:boundarycond}
\end{equation}
To consider only deformations which satisfy (\ref{def:boundarycond}), we define the functional $H_n^\ell:\A_n(0,1)\to(-\infty,+\infty]$
\begin{equation}
 H_n^\ell(u)=\begin{cases}
           H_n(u) &\text{if $u\in\A_n(0,1)$ satisfies (\ref{def:boundarycond}), }\\
	   +\infty &\text{else.}
          \end{cases}\label{def:hnl}
\end{equation}
The goal is to solve the minimization problem
\[\min_{u\in\A_n(0,1)}H_n^\ell(u),\]
which we consider as our atomistic problem.\\

The idea of energy based quasicontinuum approximations is to replace the above minimization problem by a simpler one of which minimizers and minimal energies are good approximations of the ones for $H_n^\ell$. Typically this new problem is obtained in two steps:
\begin{itemize}
 \item[(a)] Define an energy where the long range (in  our case next-to-nearest neighbour) interactions are replaced by certain nearest neighbour interactions in some regions.
 \item[(b)] Reduce the degree of freedom by choosing a smaller set of admissible functions. 
\end{itemize}
To obtain (a), the next-to-nearest neighbour interactions are approximated as
\begin{equation*}
 J_2\left(\frac{u^{i+2}-u^i}{2\lambda_n}\right)\approx \frac{1}{2}\left(J_2\left(\frac{u^{i+1}-u^i}{\lambda_n}\right)+J_2\left(\frac{u^{i+2}-u^{i+1}}{\lambda_n}\right)\right),
\end{equation*}
see e.g.\ \cite{O11}. While this approximation turns out to be appropriate in the bulk, this is not the case close to surfaces, where the second neighbour interactions create boundary layers. This motivates to construct a quasinonlocal quasicontinuum model accordingly: For given $n\in\N$ let $k_n^1,k_n^2\in\N$ with $0<k_n^1<k_n^2<n-2$. For $k_n=(k_n^1,k_n^2)$ we define the energy $\H_n^{k_n}$ by using the above approximation for $k_n^1\leq i\leq k_n^2-2$, cf. Fig.~\ref{fig:qcchain} and keeping the atomistic descriptions elsewhere 
\begin{equation*}
\begin{split}
 \H_n^{k_n}(u):=&\sum_{i=0}^{n-1}\lambda_nJ_1\left(\frac{u^{i+1}-u^i}{\lambda_n}\right)+\sum_{i=0}^{k^1_n-1}\lambda_n J_2\left(\frac{u^{i+2}-u^i}{2\lambda_n}\right)\\
 &+\sum_{i=k_n^1}^{k_n^2-2}\frac{\lambda_n}{2}\bigg\{J_2\left(\frac{u^{i+1}-u^{i}}{\lambda_n}\right)+J_{2}\left(\frac{u^{i+2}-u^{i+1}}{\lambda_n}\right)\bigg\}\\
&+\sum_{i=k^2_n-1}^{n-2}\lambda_nJ_{2}\left(\frac{u^{i+2}-u^{i}}{2\lambda_n}\right).\end{split}
\end{equation*}
Analogously to $H_n^\ell$ we define the functional $\H_n^{\ell,k_n}:\A_n(0,1)\to(-\infty,+\infty]$

\[ \H_n^{\ell,k_n}(u):=\begin{cases}
		  \H_n^{k_n}(u) &\mbox{if $u\in\A_n(0,1)$ satisfies (\ref{def:boundarycond}), }\\
		   +\infty&\text{else.}
                  \end{cases}
\]
A crucial step for the following analysis is to rewrite the energy $\H_n^{\ell,k_n}$ in a proper way. By defining
\begin{equation}\label{def:eniu}
 \mathcal E_n^i(u):=J_2\l(\frac{u^{i+2}-u^i}{2\lambda_n}\r)+\frac12\l(J_1\l(\frac{u^{i+2}-u^{i+1}}{\lambda_n}\r)+J_1\l(\frac{u^{i+1}-u^{i}}{\lambda_n}\r)\r)
\end{equation}
and  $J_{CB}(z):=J_1(z)+J_2(z)$, sometimes called Cauchy-Born energy density (see \cite{O11}), we can write
\begin{equation}\label{hnlqc}
\begin{split}
 \H_n^{\ell,k_n}(u)=&\frac{\lambda_n}{2}J_1\l(u_0^{(1)}\r)+\sum_{i=0}^{k_n^1-1}\lambda_n\mc{E}_n^i(u)+\frac{\lambda_n}{2}J_{CB}\l(\frac{u^{k_n^1+1}-u^{k_n^1}}{\lambda_n}\r)\\
& +\sum_{i=k_n^1+1}^{k_n^2-2}\lambda_nJ_{CB}\l(\frac{u^{i+1}-u^i}{\lambda_n}\r)+\frac{\lambda_n}{2}J_{CB}\l(\frac{u^{k_n^2}-u^{k_n^2-1}}{\lambda_n}\r)\\
&+\sum_{i=k_n^2-1}^{n-2}\lambda_n\mc{E}_n^i(u)+\frac{\lambda_n}{2}J_1\l(u_1^{(1)}\r),
\end{split}
\end{equation}
for $u\in\A_n(0,1)$ satisfying (\ref{def:boundarycond}). To emphasize the local structure of the continuum approximation, we rewrite the summation over the terms with $J_{CB}$ in (\ref{hnlqc}) as an integral. To this end we use the fact that $u'$ is constant on $\lambda_n(i,i+1)$ for $i=0,...,n-1$ and thus
\[\frac{\lambda_n}{2}J_{CB}\l(\frac{u^{i+1}-u^{i}}{\lambda_n}\r)=\frac12\int_{\lambda_ni}^{\lambda_n(i+1)}J_{CB}(u'(x))dx=\int_{\lambda_n(i+\frac12)}^{\lambda_n(i+1)}J_{CB}(u'(x))dx.\]
Then
\begin{equation}\label{hnlqc:integral}
\begin{split}
 \H_n^{\ell,k_n}(u)=&\frac{\lambda_n}{2}J_1\l(u_0^{(1)}\r)+\sum_{i=0}^{k_n^1-1}\lambda_n\mc{E}_n^i(u)+\int_{\lambda_n(k_n^1+\frac12)}^{\lambda_n(k_n^2-\frac12)}J_{CB}(u'(x))dx\\&+\sum_{i=k_n^2-1}^{n-2}\lambda_n\mc{E}_n^i(u)+\frac{\lambda_n}{2}J_1\l(u_1^{(1)}\r),
\end{split}
\end{equation}
for $u\in\A_n(0,1)$ satisfying (\ref{def:boundarycond}). 
 
 \begin{figure}[t]\label{fig:qcchain}
\begin{center}  \epsfig{file = 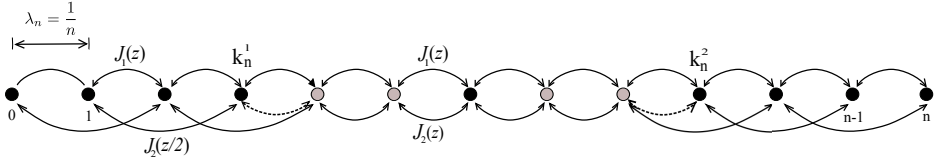, scale = 0.6}
 \end{center}
\caption{Illustration of the quasicontinuum approximation. Here $z$ denotes the scaled distance between the corresponding atoms in the deformed configuration and the two dotted lines stand for $\frac12 J_2(z)$. Moreover, the black balls symbolise the repatoms.}
%  \caption{Illustration of the quasicontinuum approximation. Here $z$ denotes the scaled distance between the corresponding atoms in the deformed configuration and the two dotted lines stand for $\frac12 J_2(z)$. Moreover, the red balls symbolise the repatoms.}
 \end{figure}
To obtain (b) we consider instead of the deformation of all atoms just the deformation of a possibly much smaller set of so called representative atoms (repatoms). We denote the set of repatoms by $\T_n =\{t_n^0,...,t_n^{r_n}\}\subset\{0,...,n\}$ with $0=t_n^0<t_n^1<...<t_n^{r_n}=n$ and define
\begin{equation}
 \A_{\T_n}(0,1):=\left\{u:[0,1]\to\mathbb{R}:u\text{ is affine on }(t_n^i,t_n^{i+1})\lambda_n\mbox{ for}~t_n^i,~t_n^{i+1}\in\T_n\right\}\label{def:atn}.
\end{equation} 
Since we are interested in the energy $\H_n^{\ell,k_n}(u)$ for deformations $u\in\A_{\T_n}(0,1)$, we define
\begin{equation}
 \H_n^{\ell,k_n,\T_n}(u):=\begin{cases}
                 \H_n^{\ell,k_n}(u) &\text{if }u\in \A_{\T_n}(0,1),\\
		 +\infty &\text{else.}
                \end{cases}\label{def:hnlqct}
\end{equation}

In the following sections we study $\H_n^{\ell,k_n,\T_n}$ as $n$ tends to infinity. Therefore, we will assume that $k_n=(k_n^1,k_n^2)$ is such that
\begin{equation}
 (i)~\lim_{n\to\infty}k_n^1=\lim_{n\to\infty}n-k_n^2=+\infty,\mbox{ and } \quad(ii)~\lim_{n\to\infty}\lambda_nk_n^1=\lim_{n\to\infty}\lambda_n(n-k_n^2)=0.\label{ass:kn}
\end{equation}
Hence, in particular $\lim_{n\to\infty}\lambda_nk_n^2=1$. The above assumption corresponds to the case that the size of the atomistic region becomes unbounded on a microscopic scale (i), but shrinks to a point on a macroscopic scale (ii). While assumption (i) is crucial, see also Remark \ref{rem:elastic}, the assumption (ii) can be easily replaced by $\lim_{n\to\infty}\lambda_nk_n^1=\xi_1$, $\lim_{n\to\infty}\lambda_n(n-k_n^2)=1-\xi_2$ and $0\leq\xi_1<\xi_2\leq1$. In this case the analysis is essentially the same, but in the case of fracture, see Theorem~\ref{theorem:fracture}, one has to distinguish more cases. We assume (\ref{ass:kn}) (ii) here because it is the canonical case from a conceptual point of view. Otherwise the atomistic region and continuum region would be on the same macroscopic scale.\\    

\section{Zero-Order $\Gamma$-Limit}

In this section we derive the $\Gamma$-limit of the discrete energy (\ref{def:hnlqct}), which we refer to as zero-order $\Gamma$-limit. This limit involves the convex and lower semicontinuous envelope $J_0^{**}$ of the effective potential energy $J_0$ which is already introduced in \cite{BGhom} defined by
\begin{equation}
 J_0(z)=J_2(z)+\frac12\inf\{J_1(z_1)+J_1(z_2):z_1+z_2=2z\}.\label{def:J_0}
\end{equation}
We state the assumptions on the functions $J_1$, $J_2$ and $J_0$ under which the following results are obtained. 

\begin{enumerate}
\item[[LJ1]\hspace*{-0.19cm}] (strict convexity) $\{z:J_{0}(z)=J_{0}^{**}(z)\}\cap\{z:J_{0}\text{ is affine near }z\}=\emptyset.$
\item[[LJ2]\hspace*{-0.19cm}] (uniqueness of minimal energy configurations) For every $z$ such that $J_{0}(z)=J_{0}^{**}(z)$ we have $\#M^z=1$ where $M^z$ is defined as
\begin{equation}
 M^z=\left\{(z_1,z_2):z_1+z_2=2z, J_0(z)=J_2(z)+\frac{1}{2}(J_1(z_1)+J_1(z_2))\right\}.\label{ass:unique1}
\end{equation}
This implies 
\begin{equation}
 J_0(z)=J_1(z)+J_2(z)=J_{CB}(z)\quad\text{for every}\quad z\in\mathbb{R}:J_0(z)=J_0^{**}(z).\label{ass:unique2}
\end{equation}
\item[[LJ3]\hspace*{-0.19cm}] (regularity and behaviour at $0$, $+\infty$). $J_1,J_2:\R\to(-\infty,+\infty]$ be in $C^{1,\alpha},0<\alpha\leq 1$ on their domains such that $J_0\in C^1$ on its domain. Let $\operatorname{dom} J_1= \operatorname{dom} J_2$ and $(0,+\infty)\subset \operatorname{dom} J_1$. Moreover, we assume the following limiting behaviour
\begin{equation}
 \lim_{z\to+\infty}J_j(z)=0,\quad j=1,2\quad\text{and}\quad\lim_{z\to+\infty}J_0(z)=J_0(+\infty)\in\R.\label{lim:j012}
\end{equation}
\item[[LJ4]\hspace*{-0.19cm}] (structure of $J_1$, $J_2$ and $J_0$). $J_1$, $J_2$ are such that there exists a convex function $\Psi:\R\to[0,+\infty]$ 
\begin{equation}\label{ass:psi1}
 \lim_{z\to-\infty}\frac{\Psi(z)}{|z|}=+\infty
\end{equation}
and there exist constants $c_1,c_2>0$ such that
\begin{equation}\label{ass:psi2}
 c_1(\Psi(z)-1)\leq J_j(z)\leq c_2\max\{\Psi(z),|z|\}\quad\mbox{for all $z\in\R$}\quad j=1,2.
\end{equation}
Further, there exist $\delta_1,\delta_2,\gamma>0$ such that
\begin{equation}\label{def:gamma}
 \{\delta_j\} =\argmin\limits_{z\in\R} J_j(z)\mbox{ for $j=1,2$ and }\{\gamma\}=\argmin\limits_{z\in\R} J_0(z)
\end{equation}
and $J_j$ is strictly convex in $(-\infty,\delta_j)$ on its domain for $j=1,2$. Moreover, it holds  $J_0(\gamma)<J_0(+\infty)$ and $J_0(z)=J_0^{**}(z)$ for all $z\leq \gamma$.\\
\end{enumerate}

\begin{remark}
(a) The main examples we think of are Lennard-Jones interactions, defined classically as
\begin{equation}
 J_1(z)=\frac{k_1}{z^{12}}-\frac{k_2}{z^6},~J_2(z)=J_1(2z),\quad\mbox{ for $z>0$ and $+\infty$ for $z\leq0$}\label{def:LJ}
\end{equation}
and $k_1,k_2>0$. The calculations in \cite[Remark 4.1]{NNN} show that $J_1,J_2$ defined as above satisfy [LJ1]--[LJ4]. Another example of interatomic potentials which satisfy the above assumptions, see \cite[Remark 4.1]{NNN}, are Morse-potentials, defined for $\delta_1,k_1,k_2>0$ as
\begin{equation}\label{def:morse}
 J_1(z)=k_1\left( 1-e^{-k_2(z-\delta_1)}\right)^2-k_1,~J_2(z)=J_1(2z),\quad\mbox{for $z\in\R$.}
\end{equation}
(b) The assumptions [LJ1]--[LJ4] imply that $J_0^{**}\equiv J_{CB}^{**}$. In particular, we have
\begin{equation}\label{j0**}
J_0^{**}(z)=\begin{cases}
               J_{CB}(z)&\mbox{if $z\leq\gamma$,}\\
	       J_{CB}(\gamma)&\mbox{if $z\geq\gamma$.}
              \end{cases} 
\end{equation}
(c) Note that [LJ4] and (\ref{lim:j012}) imply that either $\operatorname{dom} J_i=\R$ or that there exists $r_i\in\R$ such that $\operatorname{dom} J_i=(r_i,+\infty)$ or $\operatorname{dom} J_i=[r_i,+\infty)$ for $i=1,2$. In [LJ3], we assume $(0,+\infty)\subset \operatorname{dom}J_1=\operatorname{dom} J_2$ for simplicity. However, this could be dropped making suitable assumptions on $\ell,u_0^{(1)},u_1^{(1)}$ in the following statements.
\end{remark}

To define appropriate function spaces, we use a similar notation as in \cite{BCi} and \cite{NNN}. Let $u\in L_{\operatorname{loc}}^1(\R)$ be a function with bounded variation. Then we say that $u\in BV^\ell(0,1)$ if $u$ satisfies the Dirichlet boundary conditions $u(0)=0$ and $u(1)=\ell$. To allow jumps in $0$ respectively $1$, the boundary conditions are replaced by $u(0-)=0$ respectively $u(1+)=\ell$ in this case. Analogously, we define $SBV^\ell(0,1)$ for special functions with bounded variations and the above boundary conditions. Let $u\in BV^\ell(0,1)$ (or in $SBV^\ell(0,1)$), then we denote by $S_u$ the jump set of $u$ in $[0,1]$, and for $t\in S_u$ we set $[u(t)]=u(t+)-u(t-)$. Moreover we denote by $D^su$ the singular part of the measure $Du$ with respect to the Lebesgue measure.\\

Let us now state and prove the zeroth-order $\Gamma$-limit of the functional $\H_n^{\ell,k_n,\T_n}$. It turns out that the limiting functional $H^{\ell}$ is equal to the $\Gamma$-limit of the functional $H_n^\ell$, cf. \cite{NNN}.

\begin{theorem}\label{theorem:zero}
Suppose [LJ1]--[LJ4] are satisfied and let $\ell,u_0^{(1)},u_1^{(1)}>0$. 
Let $k_n=(k_n^1,k_n^2)$ satisfy (\ref{ass:kn}) and let $\T_n=\{t_n^0,...,t_n^{r_n}\}$ with $0=t_n^0<t_n^1<...<t_n^{r_n}=n$ be such that
\begin{equation}
 \exists (p_n)\subset\N \mbox{ such that }\lim_{n\to\infty}\lambda_np_n=0\mbox{ and }\sup\{t_n^{i+1}-t_n^i:t_n^{i+1},t_n^i\in\T_n\}\leq p_n. \label{T_n}
\end{equation}
 Then the $\Gamma$-limit of $H_n^\ell$ defined in (\ref{def:hnl}) and of $\H_n^{\ell,k_n,\T_n}$ defined in (\ref{def:hnlqct}) with respect to the $L^1(0,1)$--topology is the functional $H^\ell$ defined by
\[
 H^\ell(u)=
\begin{cases}\displaystyle\int_0^1  J_0^{**}(u'(x))dx & \text{if } u\in BV^\ell(0,1),~D^su\geq0,  \\
         +\infty & else,
        \end{cases}
\]
on $L^1(0,1)$.
\end{theorem}

\begin{proof}
The result for $H_n^\ell$ follows from \cite[Theorem 3.1]{NNN}. Thus we prove the result for $\H_n^{\ell,k_n,\T_n}$. The following compactness property and lower bound follow from \cite[Theorem 3.7]{BG} and \cite[Theorem 3.1]{BGhom}. For the readers convenience, we present direct proofs here.
 
\textbf{Compactness.}  Let $(u_n)$ be a sequence with equibounded energy $\H_n^{\ell,k_n,\T_n}$. The definition of $\H_n^{\ell,k_n,\T_n}$ and the properties of $J_1$, $J_2$ imply that $(u_n)\subset W^{1,\infty}(0,1)$. Define the set $I_n:=\{i\in\{0,...,n-1\}:u_n^{i+1}<u_n^i\}$. Next, we make use of the fact that $J_1,J_2$ are bounded from below and that the energy is equibounded. Moreover, we apply (\ref{ass:psi2}) and Jensen's inequality to obtain
$$C\geq \sum_{i\in I_n}\lambda_nJ_1\l(\frac{u_n^{i+1}-u_n^i}{\lambda_n}\r)\geq c_1\Psi\l(\int_{\{u_n'<0\}}u_n'dx\r)-c_1,$$
for some $C>0$ independent of $n$. By (\ref{ass:psi1}), we have that $\int_{\{u_n'<0\}}|u_n'|dx\leq C'$ for some constant $C'>0$ independent of $n$. Moreover, by using the boundary conditions, we obtain 
$$\int_{\{u_n'\geq0\}}u_n'dx=\ell-\int_{\{u_n'<0\}}u_n'dx\leq \ell+C'.$$
Since $u_n(0)=0$, we obtain by the Poincar\'e-inequality that $\|u_n\|_{W^{1,1}(0,1)}$ is equibounded. Thus, we can extract a subsequence of $(u_n)$ which converges weakly$^*$ to some $u\in BV(0,1)$, see \cite[Theorem 3.23]{AFP}. As argued in \cite[Theorem 3.1]{NNN}, we have $u\in BV^\ell(0,1)$.

\textbf{Liminf inequality.} 
Let $u\in BV^\ell(0,1)$ and $(u_n)$ be a sequence with equibounded energy $\H_n^{\ell,k_n,\T_n}$ which converges to $u$ in $L^1(0,1)$. The above compactness property and \cite[Proposition 3.13]{AFP} imply that $u_n$ converges to $u$ weakly$^*$ in $BV(0,1)$. By using [LJ3], [LJ4], we obtain for the recession function $(J_0^{**})_{\infty}$
$$(J_0^{**})_{\infty}(p):=\lim_{t\to+\infty}\frac{J_0^{**}(p_0+tp)-J_0^{**}(p_0)}{t}=
\begin{cases}
+\infty&\mbox{if $p<0$,}\\
0&\mbox{if $p\geq0$,}\end{cases}
 $$
with $p_0\in \operatorname{dom} J_0^{**}$ arbitrary. For every $\delta>0$ there exists $N\in\N$ such that $(\delta,1-\delta)\subset\lambda_n(k_n^1+\frac12,k_n^2-\frac12)$ for every $n\geq N$. For $n$ large enough, we deduce from (\ref{hnlqc:integral}) by the definition of $J_0$ and [LJ4]
$$\H_n^{\ell,k_n,\T_n}(u_n)\geq \lambda_nJ_1(\delta_1)+J_0(\gamma)|(0,1)\setminus (\delta,1-\delta)|+\int_\delta^{1-\delta}J_0^{**}(u'(x))dx$$
Note that by $(u_n)\subset W^{1,\infty}(0,1)$ it follows $D^su_n=0$ for all $n\in\N$, thus there exists $C\in\R$ such that 
\begin{align*}
 &\liminf_{n\to\infty}\H_n^{\ell,k_n,\T_n}(u_n)\\
&\geq C\delta+\liminf_{n\to\infty}\l(\int_\delta^{1-\delta}J_0^{**}(u_n'(x))dx+\int_\delta^{1-\delta}(J_0^{**})_{\infty}\l(\frac{D^su_n}{|D^su_n|}\r)d|D^su_n|\r)\\
&\geq C\delta+\int_\delta^{1-\delta}J_0^{**}(u'(x))dx+\int_\delta^{1-\delta}(J_0^{**})_{\infty}\l(\frac{D^su}{|D^su|}\r)d|D^su|.
\end{align*}
The last inequality is a direct implication of \cite[Theorem 2.34]{AFP}, using that $Du_n=u_n'\mathcal{L}^1$ weakly$^*$ converges to $Du$. By using that the right-hand side above is finite only if $D^su\geq0$, we obtain the liminf inequality from the arbitrariness of $\delta>0$.

\textbf{Limsup inequality.} To show the existence of a recovery sequence, we first do not take the boundary conditions into account. Therefore, we define the functional $\H_n^{k_n,\T_n}$ by
\begin{equation*}
 \H_n^{k_n,\T_n}(u):=\begin{cases}
                 \H_n^{k_n}(u) &\text{if }u\in \A_{\T_n}(0,1),\\
		 +\infty &\text{else.}
                \end{cases}
\end{equation*}
For every $u\in BV(0,1)$ we show existence of a sequence $(u_n)\subset L^1(0,1)$ converging to $u$ in $L^1(0,1)$ such that
\begin{equation}
 \limsup_n \H_n^{k_n,\T_n}(u_n)\leq H(u):=\int_0^1J_0^{**}(u'(x))dx.
\end{equation}
As outlined in the proof of \cite[Theorem 3.5]{BG} it is enough to show the above inequality for $u$ linear and for $u$ with a single jump: by density, this proves the statement for $u\in SBV(0,1)$ and the general estimate follows by relaxation arguments. Firstly, we consider functions $u$ with a single jump. Let $u(x)=z x + a\chi_{(x_0,1]}$  with $z\leq \gamma$, $a>0$ and $0\leq  x_0\leq 1$. By (\ref{T_n}) there exists $(h_n^1),(h_n^2)\subset \N$ with $h_n^1,h_n^2\in\T_n$ and $0< h_n^2-h_n^1\leq p_n$ such that $\lim_{n\to\infty}\lambda_nh_n^i=x_0$ for $i=1,2$. We define now a sequence $(u_n)$ by
\begin{equation}\label{rec:crack}
 u_n^i=\begin{cases}
         zi\lambda_n&\mbox{if $0\leq i\leq h_n^1$,}\\
	 zi\lambda_n+a\frac{i-h_n^1}{h_n^2-h_n^1}&\mbox{if $h_n^1\leq i\leq h_n^2$,}\\
	 zi\lambda_n+a&\mbox{if $h_n^2\leq i\leq n$.}
        \end{cases}
\end{equation}
Obviously we have $u_n\to u$ in $L^1(0,1)$. The functions $u_n$ are defined such that $u_n^{i+1}-u_n^i=\lambda_nz+\frac{a}{h_n^2-h_n^1}$ for $i\in\{h_n^1,...,h_n^2-1\}$ and $u_n^{i+1}-u_n^i=\lambda_nz$ for all $0\leq i<n$ with $i\notin\{h_n^1,...,h_n^2-1\}$. Using $h_n^2-h_n^1\leq p_n$, (\ref{T_n}), [LJ3] and [LJ4] this implies
$$ \H_n^{k_n,\T_n}(u_n)=J_1(z)+J_2(z)+\mc{O}(\lambda_np_n)\to\int_0^1J_0^{**}(z)dx\quad\mbox{as $n\to\infty$}.$$
Now let $u(x)=zx$ for some $z>\gamma$. For every sequence $(p_n)$ satisfying (\ref{T_n}) we find a sequence $(q_n)$ of natural numbers such that
\[\lim_{n\to\infty}\lambda_nq_n=0,\quad \lim_{n\to\infty}\frac{p_n}{q_n}=0.\]
We define for every $n\in\N$ a set $\T_n'\subset\T_n:=\{t_n^0,...,t_n^{r_n}\}$ with $\T_n'=\{t_n^{h_n^0},...,t_n^{h_n^{N_n}}\}$, where $0=h_n^0<h_n^1<...<h_n^{N_n}=r_n$ such that there exist $c_1,c_2>0$ which satisfy
\[c_1q_n\leq t_n^{h_n^{j+1}}-t_n^{h_n^j}\leq c_2q_n\qquad \forall j\in\{0,...,N_n-1\}.\]
From $n=\sum_{j=0}^{N_n-1}\l(t_n^{h_n^{j+1}}-t_n^{h_n^j}\r)$ we deduce $c_1N_nq_n\leq n\leq c_2N_nq_n$ and thus $N_nq_n=\mc{O}(n)$. Let us define $u_n\in\A_{\T_n}(0,1)$ such that $u_n(1)=z$ and
\[u_n(x)=z\lambda_nt_n^{h_n^j}+\gamma(x-\lambda_nt_n^{h_n^j})\quad\mbox{for $x\in[t_n^{h_n^j},t_n^{h_n^{j+1}-1}]\lambda_n$ and $j\in\{0,...,N_n-1\}$.}\]
By using $t_n^{h_n^j}-t_n^{h_n^j-1}\leq p_n$ $\forall j\in\{1,...,N_n\}$ and $|u(x)-u_n(x)|\leq2z$, we obtain
\begin{align*}
\int_0^1|u(x)-u_n(x)|dx=&\sum_{j=0}^{N_n-1}\int_{\lambda_nt_n^{h_n^j}}^{\lambda_nt_n^{h_n^{j+1}-1}}\Big|zx-z\lambda_nt_n^{h_n^j}-\gamma\l(x-\lambda_nt_n^{h_n^j}\r)\Big|dx\\
 &+\sum_{j=1}^{N_n}\int_{\lambda_nt_n^{h_n^j-1}}^{\lambda_nt_n^{h_n^{j}}}|u(x)-u_n(x)|dx\\
\leq &\sum_{j=0}^{N_n-1}\int_{\lambda_nt_n^{h_n^j}}^{\lambda_nt_n^{h_n^{j+1}-1}}(z-\gamma)(x-\lambda_nt_n^{h_n^j})dx+2zN_n\lambda_np_n\\
= &\sum_{j=0}^{N_n-1}\frac12(z-\gamma)\lambda_n^2\l(t_n^{h_n^{j+1}-1}-t_n^{h_n^j}\r)^2+2z N_n\lambda_np_n\\
\leq&\frac12(z-\gamma)N_nc_2^2q_n^2\lambda_n^2+2z \lambda_np_nN_n
\end{align*}
and thus $u_n\to u$ in $L^1(0,1)$. Indeed, by $\lambda_nN_nq_n=\mc{O}(1)$, $\lambda_nq_n\to0$ and $\mc{O}(\lambda_np_nN_n)=\mc{O}\l(\frac{p_n}{q_n}\r)$, the last term tends to zero as $n\to\infty$. For the limsup inequality we argue similarly as in the case of a jump before. By definition, we have $u_n^{i+1}-u_n^i=\lambda_n\gamma$ for $0\leq i\leq n-1$ and $i\notin\l(\N\cap\cup_{j=1}^{N_n}[t_n^{h_n^j-1}, t_n^{h_n^j})\r)$ and by using $\#\l(\N\cap\cup_{j=1}^{N_n}[t_n^{h_n^j-1}, t_n^{h_n^j})\r)\leq N_np_n$, we have  
$$ \H_n^{k_n,\T_n}(u_n)=J_1(\gamma)+J_2(\gamma)+\mc{O}(\lambda_np_nN_n). $$
Since $\lambda_np_nN_n\to0$ as $n\to\infty$ we deduce, using (\ref{j0**}), the limsup inequality in this case. Combining the arguments we have the limsup inequality for all functions which are linear except in a single jump.\\
Now let $u\in BV^\ell(0,1)$ with $H^\ell(u)<+\infty$. The above procedure and similar arguments as in \cite[Theorem 3.1]{BCi} provides a sequence $(u_n)$ which satisfies $u_n^0=0$ and $u_n^n=\ell$ but not necessarily satisfies the boundary conditions on the second and last but one atom. In general it is not clear if for example $\frac{1}{\lambda_n}\l( u_n^2- \lambda_n u_0^{(1)}\r)\in \operatorname{dom}J_1$ for all $n\in\N$. Thus, we cannot simply replace $u_n^1$ or $u_n^{n-1}$ by the given boundary conditions. We show now how to overcome this. As before, it is sufficient to show the limsup inequality for functions $u\in BV^\ell(0,1)$ which are piecewise affine with positive jumps. From $\ell>0$, we deduce that $\#S_u\geq1$ or $u'>0$ on some open interval $I\subset[0,1]$. Firstly, we assume that there exists $x\in[0,1]$ with $x\in S_u$. Without loss of generality, we can assume that $(u_n)$ satisfies $u_n^1-u_n^0=\mathcal O(\lambda_n)$ and $\ell-u_n^{n-1}=\mathcal O(\lambda_n)$ as $n\to\infty$. As in the sequence constructed in (\ref{rec:crack}), there exist $(h_n^1),(h_n^2)\subset\N$ with $h_n^1<h_n^2\in\T_n$ and $\lim_{n\to\infty}\lambda_nh_n^i=x\in[0,1]$ for $i=1,2$ and $\T_n\cap\{h_n^1+1,..,h_n^2-1\}=\emptyset$ for all $n\in\N$ such that 
$$\lim_{n\to\infty}\l( u_n^{h_n^2}-u_n^{h_n^1}\r)=[u](x)>0.$$  
Define now $(\tilde u_n)$ such that $\tilde u_n\in \A_{\T_n}(0,1)$ and 
\begin{equation}\label{bcrecovery}
 \tilde u_n^i=\begin{cases}
   0&\mbox{if $i=0$},\\
   u_n^i+\lambda_nu_0^{(1)}-u_n^1&\mbox{if $1\leq i\leq h_n^1$},\\
   \tilde u_n^{h_n^1}+\frac{i-h_n^1}{h_n^2-h_n^1}(\tilde u_n^{h_n^2}-\tilde u_n^{h_n^1})&\mbox{if $h_n^1+1\leq i\leq h_n^2-1$,}\\
   u_n^i+\ell-\lambda_nu_1^{(1)}-u_n^{n-1}&\mbox{if $h_n^2\leq i\leq n-1$},\\ 
   \ell&\mbox{if }i=n.
  \end{cases}
\end{equation}
Then $\tilde u_n$ satisfies the boundary conditions and we have $\|u_n-\tilde u_n\|_{L^1(0,1)}\to0$ as $n\to\infty$ and thus $\tilde u_n\to u$ in $L^1(0,1)$. Moreover, we have $\tilde u'_n \equiv u'_n$ on $\lambda_n\l( (1,h_n^1)\cup (h_n^2,n-1)\r)$ and 
\begin{equation}\label{zero:utilde}
 \tilde u_n^{h_n^2}- \tilde u_n^{h_n^1}=u_n^{h_n^2} - u_n^{h_n^1} + \ell - u_n^{n-1}-\lambda_n(u_0^{(1)}+u_1^{(1)})+u_n^1=u_n^{h_n^2}-u_n^{h_n^1}+\mathcal O(\lambda_n)\to[u](x)
\end{equation}
as $n\to\infty$. Thus $\tilde u_n$ is a recovery sequence for $u$. \\
Let now $u'\equiv z>0$ on some open interval $I\subset[0,1]$. There exist $(h_n^1),(h_n^2)\subset\N$ with $h_n^1<h_n^2\in\T_n$ and $\lim_{n\to+\infty}\l( h_n^2-h_n^1 \r) =+\infty$ and $\lim_{n\to\infty}\lambda_n(h_n^2-h_n^1)=0$ with $\lambda_n(h_n^1,h_n^2)\subset I$. We define now $(\tilde u_n)$ as in (\ref{bcrecovery}). As above, we have $\tilde u_n\to u$ in $L^1(0,1)$ and $\tilde u'_n \equiv u'_n$ on $\lambda_n\l( (1,h_n^1)\cup (h_n^2,n-1)\r)$. By (\ref{zero:utilde}), we have for all $t\in\lambda_n(h_n^1,h_n^2)$
$$\tilde u_n'(t) = u_n'(t)+\mathcal{O}((h_n^2-h_n^1)^{-1})>0, $$
for $n$ large enough. Using $\lim_{n\to\infty} \l( h_n^2-h_n^1 \r)=+\infty$ and [LJ3] implies that the sequence $\tilde u_n$ is a recovery sequence for $u$.   
\end{proof}

\begin{remark}\label{rem:zero} (a) Jensen's inequality implies $\min_u H^\ell(u)=J_0^{**}(\ell)$ for every $\ell$.\\
(b) The $\Gamma$-limit of zeroth order computed in Theorem~\ref{theorem:zero} does not give any information about boundary layer energies or the number and location of possible jumps. Thus we need to compare the functionals $H_n^\ell$ and $\H_n^{\ell,k_n,\T_n}$ at a higher order in $\lambda_n$, which will be done in the next section. To underline that the zeroth-order $\Gamma$-limit is too coarse to measure the quality of the quasicontinuum method, we remark that one can show that the functional defined as
\[H_n^{\ell,CB}(u):=\begin{cases}
                  \sum_{i=0}^{n-1}\lambda_nJ_{CB}\l(\frac{u^{i+1}-u^i}{\lambda_n}\r)&\mbox{if $u\in\mc{A}_n(0,1)$ satisfies (\ref{def:boundarycond}),}\\
		  +\infty&\mbox{else,}
                 \end{cases}
\]
$\Gamma$-converges to $H^\ell$ with respect to the strong topology of $L^1(0,1)$. Note that $H_n^{\ell,CB}$ can be understood as a continuum approximation of $H_n^\ell$.

\end{remark}

\section{First order $\Gamma$-Limit}

In this section, we derive the $\Gamma$-limit of the functional $\H_{1,n}^{\ell,k_n,\T_n}$ defined by
\begin{equation}
\H_{1,n}^{\ell,k_n,\T_n}(u)=\frac{\H_n^{\ell,k_n,\T_n}(u)-\min_v H^\ell(v)}{\lambda_n},\label{def:hn1qc}
\end{equation}
which is called the $\Gamma$-limit of first order. In \cite{NNN}, this is done for the functionals $H_{1,n}^\ell(u)=\frac{1}{\lambda_n}\l(H_n^\ell(u)-\min_v H^\ell(v)\r)$ and in \cite{BCi} for a similar functional; we can use several ideas from there for our setting. To shorten the notation, we omit the index $\T_n$ of $\H_{1,n}^{\ell,k_n,\T_n}$ if we consider $(\T_n)$ such that $\T_n=\{0,...,n\}$ for all $n\in\N$.\\
It will be useful to rearrange the terms in the expression of the energy $\H_{1,n}^{\ell,k_n,\T_n}$ in a similar way as in \cite{BCi} or \cite{NNN}: For given $\ell,u_0^{(1)},u_1^{(1)}>0$ let $(u_n)$ be a sequence of functions satisfying the boundary conditions (\ref{def:boundarycond}) for each $n$. We obtain from Remark~\ref{rem:zero} (a), (\ref{def:hn1qc}) and (\ref{hnlqc}) by adding and subtracting $\sum_{i=0}^{n-2}(J_0^{**})'(\ell)\l(\frac{u_n^{i+2}-u_n^i}{2\lambda_n}-\ell\r)$
\begin{align*}
 \H_{1,n}^{\ell,k_n,\T_n}(u_n)=&\frac{1}{2}J_1\left(u_0^{(1)}\right)+\sum_{i=0}^{k_n^1-1}\bigg\{\E_n^i(u_n)-J_0^{**}(\ell)-(J_0^{**})'(\ell)\left(\frac{u_n^{i+2}-u_n^i}{2\lambda_n}-\ell\right)\bigg\}\\
 &\quad -(J_0^{**})'(\ell)\left(\frac{u_n^{k_n^1+2}-u_n^{k_n^1}}{2\lambda_n}-\ell\right)+\frac{1}{2}J_{CB}\left(\frac{u_n^{k_n^1+1}-u_n^{k_n^1}}{\lambda_n}\right)\\
 &+\sum_{i=k_n^1+1}^{k_n^2-2}\left(J_{CB}\left(\frac{u_n^{i+1}-u_n^i}{\lambda_n}\right)-J_0^{**}(\ell)-(J_0^{**})'(\ell)\left(\frac{u_n^{i+2}-u_n^i}{2\lambda_n}-\ell\right)\right)\\
 &+\frac{1}{2}J_{CB}\left(\frac{u_n^{k_n^2}-u_n^{k_n^2-1}}{\lambda_n}\right)+\sum_{i=k_n^2-1}^{n-2}\bigg\{\E_n^i(u_n)-J_0^{**}(\ell)\\
 &-(J_0^{**})'(\ell)\left(\frac{u_n^{i+2}-u_n^i}{2\lambda_n}-\ell\right)\bigg\}-2J_0^{**}(\ell)+\frac12J_1(u_1^{(1)})\\
 &+\sum_{i=0}^{n-2}(J_0^{**})'(\ell)\left(\frac{u_n^{i+2}-u_n^i}{2\lambda_n}-\ell \right).
\end{align*}
Since
$$ \sum_{i=0}^{n-2}(u_n^{i+2}-u_n^i)=2\sum_{i=0}^{n-1}(u_n^{i+1}-u_n^i)-(u_n^1-u_n^0)-(u_n^n-u_n^{n-1})=2\ell-\lambda_n\l(u_0^{(1)}+u_1^{(1)}\r),$$
and \cite[(4.16)]{NNN}, the last term reads
\[\sum_{i=0}^{n-2}(J_0^{**})'(\ell)\left(\frac{u_n^{i+2}-u_n^i}{2\lambda_n}-\ell\right)=-(J_0^{**})'(\ell)\left(\frac{u_0^{(1)}+u_1^{(1)}}{2}-\ell\right).\]
In the same way we can rewrite the terms containing the sum over $k_n^1+1\leq i\leq k_n^2-2$ by
\begin{eqnarray*}
& &\sum_{i=k_n^1+1}^{k_n^2-2}\left(J_{CB}\left(\frac{u_n^{i+1}-u_n^i}{\lambda_n}\right)-J_0^{**}(\ell)-(J_0^{**})'(\ell)\left(\frac{u_n^{i+2}-u_n^i}{2\lambda_n}-\ell\right)\right)\\
& &=\sum_{i=k_n^1+1}^{k_n^2-2}\left(J_{CB}\left(\frac{u_n^{i+1}-u_n^i}{\lambda_n}\right)-J_0^{**}(\ell)-(J_0^{**})'(\ell)\left(\frac{u_n^{i+1}-u_n^i}{\lambda_n}-\ell\right)\right)\\
& &\quad+\frac{1}{2}(J_0^{**})'(\ell)\left(\left(\frac{u_n^{k_n^1+2}-u_n^{k_n^1+1}}{\lambda_n}-\ell\right)-\left(\frac{u_n^{k_n^2}-u_n^{k_n^2-1}}{\lambda_n}-\ell\right)\right).
\end{eqnarray*}
Let $(u_n)$ be such that $u_n\in\A_n(0,1)$, then we define 
\begin{equation}
\sigma_n^i(\ell):=\E_n^i(u_n)-J_0^{**}(\ell)-(J_0^{**})'(\ell)\left(\frac{u_n^{i+2}-u_n^i}{2\lambda_n}-\ell\right),\label{def:sigmal}
\end{equation}
with $\E_n^i(u_n)$ defined in (\ref{def:eniu}) and 
\begin{equation}
\mu_n^i(\ell):=J_{CB}\left(\frac{u_n^{i+1}-u_n^i}{\lambda_n}\right)-J_0^{**}(\ell)-(J_0^{**})'(\ell)\left(\frac{u_n^{i+1}-u_n^i}{\lambda_n}-\ell\right).\label{def:mul} 
\end{equation}
By using the definition of $J_0$ and $J_{CB}$, we have $J_{CB}(z)\geq J_0(z)\geq J_0(\gamma)$ which implies with (\ref{def:J_0}) and $J_0^{**}(z)=J_0(\gamma)$ for $z\geq\gamma$ that $\sigma_n^i(\ell)=\sigma_n^i(\gamma),\mu_n^i(\ell)=\mu_n^i(\gamma)\geq0$ for $\ell\geq\gamma$ and we will often drop the variable $\ell$ in this case and write $\sigma_n^i$ and $\mu_n^i$ for short. For $\ell\leq\gamma$, we have
\[J_{0}(z)-J_0^{**}(\ell)-(J_0^{**})'(\ell)(z-\ell)\geq J_{0}^{**}(z)-J_0^{**}(\ell)-(J_0^{**})'(\ell)(z-\ell)\geq 0\]
for all $z\in\R$ and from $J_{CB}(z)\geq J_0(z)$ and $J_{CB}^{**}\equiv J_0^{**}$ we deduce $\sigma_n^i(\ell),\mu_n^i(\ell)\geq0$.\\
We can now rewrite $\H_{1,n}^{\ell,k_n,\T_n}(u_n)$ such that all unknowns $u_n^i$, $i=2,...,n-2$ are arranged in non-negative terms
\begin{align}
 \H_{1,n}^{\ell,k_n,\T_n}(u_n)=&\frac{1}{2}J_1\left(u_0^{(1)}\right)+\sum_{i=0}^{k^1_n-1}\sigma_n^i(\ell)-(J_0^{**})'(\ell)\left(\frac{u_n^{k^1_n+2}-u_n^{k^1_n}}{2\lambda_n}-\ell\right)\notag\\
 &+\frac{1}{2}J_{CB}\left(\frac{u_n^{k_n^1+1}-u_n^{k_n^1}}{\lambda_n}\right)+\sum_{i=k_n^1+1}^{k_n^2-2}\mu_n^i(\ell)\notag\\
&+\frac{1}{2}(J_0^{**})'(\ell)\left(\left(\frac{u_n^{k_n^1+2}-u_n^{k_n^1+1}}{\lambda_n}-\ell\right)-\left(\frac{u_n^{k_n^2}-u_n^{k_n^2-1}}{\lambda_n}-\ell\right)\right)\notag\\
 &+\frac{1}{2}J_{CB}\left(\frac{u_n^{k_n^2}-u_n^{k_n^2-1}}{\lambda_n}\right)+\sum_{i=k_n^2-1}^{n-2}\sigma_n^i(\ell)+\frac{1}{2}J_1(u_1^{(1)})-2J_0^{**}(\ell)\notag\\
 &-(J_0^{**})'(\ell)\left(\frac{u_0^{(1)}+u_1^{(1)}}{2}-\ell\right)\notag\\
=&\frac{1}{2}J_1\left(u_0^{(1)}\right)+\sum_{i=0}^{k_n^1-1}\sigma_n^i(\ell)+\frac{1}{2}\mu_n^{k_n^1}(\ell)+\sum_{i=k_n^1+1}^{k_n^2-2}\mu_n^i(\ell)+\frac{1}{2}\mu_n^{k_n^2-1}(\ell)\notag\\
 &+\sum_{i=k_n^2-1}^{n-2}\sigma_n^i(\ell)+\frac{1}{2}J_1\left(u_1^{(1)}\right)-J_0^{**}(\ell)-(J_0^{**})'(\ell)\left(\frac{u_0^{(1)}+u_1^{(1)}}{2}-\ell\right).\label{def:hnl1<}
\end{align}

Before we state the compactness results about sequences $(u_n)$ with equibounded energies $H_{1,n}^\ell$ and $\H_{1,n}^{\ell,k_n,\T_n}$, we prove the following lemma.
\begin{lemma}\label{lemma:F}
 Let $\ell>0$ and $J_1,J_2$ satisfy [LJ1]--[LJ4]. Let $\eps>0$. Then there exists $\eta=\eta(\eps)>0$ such that
\begin{equation}\label{ineq:F}
 F(z):=\inf_{a:|a-\min\{\ell,\gamma\}|\geq\eps}\frac12\left(J_1(a)+J_1(2z-a)\right)+J_2(z) - (J_0^{**})'(\ell)\left(z-\ell\right)-J_0^{**}(\ell)\geq \eta.
\end{equation}
\end{lemma}

\begin{proof}
We distinguish between the cases when $z$ is close to $\min\{\ell,\gamma\}$ or not. Let us first define the function $\tilde J(a,z):=\frac12(J_1(a)+J_1(2z-a))$. Clearly $\tilde J$ is continuous on its domain. If $z$ and $\eps>0$ are such that $\inf_{a:|z-a|\geq\eps}\tilde J(a,z)=+\infty$, inequality (\ref{ineq:F}) holds trivially. Thus, we can assume that $\inf_{a:|z-a|\geq\eps}\tilde J(a,z)$ is finite. From the growth conditions of $J_1$ at $-\infty$, we deduce that for given $z\in\R$, $\eps>0$ the infimum problem $\inf_{a:|z-a|\geq\eps}\tilde J(a,z)$ attains its minimum. 
Furthermore, the assumption [LJ2] and [LJ4] imply that there exists $\eta_1=\eta_1(z,\eps)>0$ such that
\begin{equation}\label{feta1}
 \min_{a:|z-a|\geq\eps}\tilde J(a,z)+J_2(z)-J_0^{**}(z)\geq \eta_1>0.
\end{equation}
The function $f(z):=\min_{a:|a-z|\geq\eps}\tilde J(a,z)$ is lower semicontinuous. Indeed, this can be proven by using the growth conditions of $J_1$.  
%Indeed, let $z_k,z\in \{z:|z-\min\{\ell,\gamma\}|\leq\eps\}\cap \operatorname{dom}f $, $k\in\N$ be such that $\lim_{k\to\infty}z_k=z$. From the considerations above, we deduce that there exist $(a_k)\subset\R$ and $b\in\R$ such that $|a_k-z_k|\geq\eps$, $|z-b|\geq\eps$ and $f(z_k)=\tilde J(a_k,z_k)$, $f(z)=\tilde J(b,z)$. Moreover, by the growth conditions of $J_1$ and the definition of $\tilde J$, there exists $C>0$ such that $\sup_{k\in\N}|a_k|<C$. There exists a subsequence $z_{k_j}\subset (z_k)$ such that $\liminf_{k\to\infty}f(z_k)=\lim_{j\to\infty}f(z_{k_j})=\lim_{j\to\infty}\tilde J(a_{k_j},z_{k_j})$. Since $\sup_k|a_k|<C$, we can extract a further converging subsequence such that  $\lim_{i\to\infty}a_{{k_j}_i}=a\in \R$ and from $|a_k-z_k|\geq \eps$ and $z_k\to z$ as $k\to\infty$, we deduce $|a-z|\geq\eps$. Furthermore, let $(b_k)\subset\R$ be such that $b_k\to b$ and $|z_k-b_k|\geq \eps$. Note that $f(z_k)\leq \tilde J(b_k,z_k)$ and thus
% \begin{align*}
%  f(z)\leq& \tilde J(a,z)=\lim_{i\to\infty}\tilde J\left(a_{{k_j}_i},z_{{k_j}_i}\right)=\liminf_{k\to\infty}f(z_k)\leq \limsup_{k\to\infty}f(z_k)\leq \limsup_{k\to\infty}\tilde J(b_k,z_k)\\
% =&\tilde J(b,z)=f(z). 
% \end{align*}
%Hence $f(z)=\min_{a:|a-z|\geq\eps}\tilde J(a,z)$ is continuous on $\{z:|z-\min\{\ell,\gamma\}|\leq\eps\}$. 
Thus, we deduce from inequality (\ref{feta1}) that there exists $\eta_2=\eta_2(\eps)>0$ such that
$$\inf_{z:|z-\min\{\ell,\gamma\}|\leq\eps}\left\{\min_{a:|z-a|\geq\eps}\tilde J(a,z)+J_2(z)-J_0^{**}(z)\right\}\geq \eta_2>0.$$
Let now $|z-\min\{\ell,\gamma\}|\leq \frac{\eps}{2}$. Since $|a-\min\{\ell,\gamma\}|\geq \eps$ implies $|a-z|\geq \frac{\eps}{2}$, we have
$$F(z)\geq J_0^{**}(z)+\eta_2\left(\frac{\eps}{2}\right)- (J_0^{**})'(\ell)\left(z-\ell\right)-J_0^{**}(\ell)\geq\eta_2\left(\frac{\eps}{2}\right).$$
It is left to consider the case $|z-\min\{\ell,\gamma\}|\geq\frac{\eps}{2}$. By the definition of $J_0$, we have
$$F(z)\geq \min_{z:|z-\min\{\ell,\gamma\}|\geq\frac{\eps}{2}}J_0(z)-(J_0^{**})'(\ell)\left(z-\ell\right)-J_0^{**}(\ell)=:\eta_3(\eps)>0.$$
Indeed, the existence of $\eta_3$ as above follows from the strict convexity of $J_0$ on $(-\infty,\gamma)$, that $\gamma$ is the unique minimizer of $J_0$ and $\lim_{z\to\infty}J_0(z)=J_0(\infty)>J_0(\gamma)$. Altogether, the assertion is proven with $\eta(\eps)=\min\left\{\eta_2\left(\frac{\eps}{2}\right),\eta_3(\eps)\right\}$.
\end{proof}

We are now in position to state a compactness result analogously to \cite[Proposition 4.2]{BCi} and \cite[Proposition 4.1]{NNN}. 
\begin{proposition}\label{prop:compact}
 Let $\ell,u_0^{(1)},u_1^{(1)}>0$ and suppose that hypotheses [LJ1]--[LJ4] hold. Let $(k_n)=(k_n^1,k_n^2)$ satisfy (\ref{ass:kn}) and let $(u_n)$ be a sequence of functions such that
\begin{equation}
  \sup_n \H_{1,n}^{\ell,k_n,\T_n}(u_n)<+\infty.\label{eq:prop}
\end{equation}
(1) If $\ell\leq \gamma$, then, up to subsequences, $u_n\to u$ in $L^\infty(0,1)$ with $u(x)=\ell x$, $x\in[0,1]$.\\
(2) In the case $\ell>\gamma$, then, up to subsequences, $u_n\to u$ in $L^1(0,1)$ where $u\in SBV^\ell(0,1)$ is such that
\begin{enumerate}
 \item[(i)] $0<\#S_u<+\infty$;% for $\ell>\gamma$ and $S_u=\emptyset$ for $\ell\leq\gamma$;
 \item[(ii)] $[u]>0$ on $S_u$;
 \item[(iii)] $u'=\gamma$ a.e.%$\min\{\gamma,\ell\}$ a.e.
\end{enumerate}
\end{proposition}

\begin{proof} 
Let $(u_n)$ satisfy (\ref{eq:prop}). With the same arguments as in the proof of Theorem~\ref{theorem:zero}, we have the existence of $u\in BV^\ell(0,1)$ such that, up to subsequences, $u_n\to u$ weakly$^*$ in $BV(0,1)$.\\
Let us show $u_n'\to\min\{\ell,\gamma\}$ in measure in $(0,1)$. For $\eps>0$, we define $$I_n^\eps:=\left\{i\in\{0,...,n-1\}:\left| \frac{u_n^{i+1}-u_n^i}{\lambda_n}-\min\{\ell,\gamma\}\right|>\eps\right\}.$$
By the definition of $\sigma_n^i(\ell)$, $\mu_n^i(\ell)$, see (\ref{def:sigmal}), (\ref{def:mul}), and Lemma~\ref{lemma:F}, we deduce the existence of $\eta=\eta(\eps)>0$ such that $\sigma_n^i(\ell),\mu_n^i(\ell)\geq\eta$ for $i\in I_n^\eps$.  
By (\ref{eq:prop}), there exists a constant $C>0$ such that 
\[C\geq \sum_{i=0}^{k_n^1-1}\sigma_n^i(\ell)+\sum_{i=k_n^1+1}^{k_n^2-2}\mu_n^i(\ell)+\sum_{i=k_n^2-1}^{n-2}\sigma_n^i(\ell)\geq \#I_n^\eps\eta.\]
Hence, by using $|\{x:|u_n'(x)-\min\{\ell,\gamma\}|>\eps\}|=\lambda_n\#I_n^\eps\leq \lambda_n\frac{C}{\eta}$ it follows that $u_n'\to\min\{\ell,\gamma\}$ in measure. Moreover, we can use the above argument in the following way: we define the set 
$$Q_n:=\left\{i\in \{0,...,n-2\}:\frac{u_n^{i+1}-u_n^i}{\lambda_n}>2\gamma \right\}.$$
As above, Lemma~\ref{lemma:F} ensures $\sigma_n^i(\ell),\mu_n^i(\ell)\geq\eta$ for $i\in Q_n$ and some $\eta>0$. From (\ref{eq:prop}), we deduce the equiboundedness of $\#Q_n$.  We define the sequence $(v_n)\subset SBV^\ell(0,1)$ as
\[v_n(x)=\begin{cases}
         u_n(x), &\text{if }x\in(i,i+1)\lambda_n,i\notin Q_n,\\
	 u_n(i\lambda_n), &\text{if }x\in(i,i+1)\lambda_n,i\in Q_n.
        \end{cases}\]
The sequence $(v_n)$ is constructed such that $\lim_{n\to\infty}\int_0^1|u_n-v_n|dx=0$ and thus we can assume, by passing to a subsequence, that $(v_n)$ converges to $u$ in the weak$^*$ topology of $BV(0,1)$. By definition of $v_n$, we have $\#S_{v_n}= \#Q_n$ and thus there exists a constant $C>0$ such that $\sup_n\#S_{v_n}\leq C$. Using $v_n'(x)\leq2\gamma$ a.e., (\ref{ass:psi1}) and (\ref{ass:psi2}), the sequence $(v_n)$ satisfies all assumptions of \cite[Theorem 4.7]{AFP} and we conclude that $u\in SBV^\ell(0,1)$, $v_n'\rightharpoonup u'$ weakly in $L^1(0,1)$, $\liminf_{n\to\infty}\#S_{v_n}\geq \#S_u$ and $D^jv_n$ weakly$^*$ converge to $D^ju$, where $D^jv$ denotes the jump part of the derivative of $v\in BV(\R)$. As a direct consequence, we obtain $\#S_u<+\infty$. By the construction of $(v_n)$, we have $[v_n]>0$ on $S_{v_n}$ and we conclude, by the weak$^*$ convergence of the jump part, assertion (ii).\\
Note that $(v_n)$ is defined such that $|\{x:u_n'(x)\neq v_n'(x)\}|\leq \# S_{v_n}\lambda_n$, which implies $v_n'\to\min\{\ell,\gamma\}$ in measure in $(0,1)$. Combining this with $v_n'\rightharpoonup u'$ in $L^1(0,1)$, we show $u'=\min\{\ell,\gamma\}$ a.e.\ in $(0,1)$. Indeed, by the Dunford-Pettis theorem, we deduce from the relative compactness of $(v_n')\subset L^1(0,1)$ in the weak $L^1(0,1)$--topology that $(v_n')$ is equi-integrable. By extracting a subsequence, we can assume that $v_n'\to\min\{\ell,\gamma\}$ pointwise a.e.\ in $(0,1)$ and by Vitali's convergence theorem it follows $v_n'\to \min\{\ell,\gamma\}$ strongly in $L^1(0,1)$. Thus $u'=\min\{\ell,\gamma\}$ a.e.\ in $(0,1)$. Thus the assertion for $\ell>\gamma$ is proven. In the case $0<\ell\leq\gamma$, we have, up to subsequences, $u_n\to u$ in $L^1(0,1)$ with $u\in BV^\ell(0,1)$, $u'=\ell$ a.e.\ in $(0,1)$ and $[u]>0$ on $S_u$. This implies $u(x)=\ell x$ on $[0,1]$. It is left to show: $u_n\to u$ in $L^\infty(0,1)$. Note that for the above defined sequence $(v_n)$ it holds $u_n'=v_n'+w_n$ a.e. on $(0,1)$ with $w_n\in L^1(0,1)$ and $w_n(x)\geq 0$. Using $v_n'\to\ell$ in $L^1(0,1)$, we deduce from
$$\ell=\int_0^1u_n'(x)dx=\int_0^1v_n'(x)dx+\int_0^1w_n(x)dx$$
that $w_n\to0$ in $L^1(0,1)$. Altogether, we have $u_n'=v_n'+w_n\to \ell$ in $L^1(0,1)$ and thus $u_n\to u$ in $W^{1,1}(0,1)$ with $u(x)=\ell x$. Hence, the assertion follows from the Sobolev inequality on intervals.
\end{proof}

For $\ell>\gamma$ we define the space
\begin{equation}\label{def:sbvcl}
 SBV_c^\ell(0,1):=\{u\in SBV^\ell(0,1)\mbox{ : conditions (i)-(iii) of Proposition~\ref{prop:compact} are satisfied}\},
\end{equation}
as in \cite{NNN}. 
 
Proposition~\ref{prop:compact} tells us that a sequence of deformations $(u_n)$ with equibounded energy converges in $L^1(0,1)$ to a deformation $u$ which has a constant gradient almost everywhere. In the following lemma, we prove that $(u_n)$ yields a sequence of discrete gradients in the atomistic region converging to the same constant. This turns out to be crucial in the proofs of the first order $\Gamma$-limits.
\begin{lemma}\label{lemma:kngegen0}
Suppose that hypotheses [LJ1]--[LJ4] hold. Let $\ell,u_0^{(1)},u_1^{(1)}>0$. Let $(u_n)$ be a sequence of functions such that (\ref{eq:prop}) is satisfied. Let $(k_n)=(k_n^1,k_n^2)$ satisfy (\ref{ass:kn}). Then there exist sequences $(h_n^1),(h_n^2)\subset\N$ with $0\leq h_n^1<k_n^1-2<k_n^2+2<h_n^2\leq n-1$ such that, up to subsequences,
\begin{equation}
 \lim_{n\to\infty} \frac{u_n^{h^i_n+1}-u_n^{h^i_n}}{\lambda_n}=\min\{\ell,\gamma\},\qquad i=1,2.\label{claim:kngegen0}
\end{equation}
\end{lemma}
\begin{proof}
Let us define $(\tilde k_n)\subset \N$ by $\tilde k_n=\min\{k_n^1,n-k_n^2\}$ and 
\begin{equation*}
I_n:=\{i\in\{0,...,k_n^1-1,k_n^2-1,....,n-2\}:\sigma_n^i(\ell)>\frac{1}{\sqrt{\tilde k_n}}\}.\end{equation*}
By (\ref{eq:prop}) there exists $C>0$ such that 
\begin{equation*}
C\geq\sup_n \l(\sum_{i=0}^{k_n^1-1}\sigma_n^i(\ell)+\sum_{i=k^2_n-1}^{n-2}\sigma_n^i(\ell)\r)\geq\sup_n\sum_{i\in I_n} \frac{1}{\sqrt{\tilde k_n}}=\sup_n\frac{\#I_n}{\sqrt{\tilde k_n}}.
\end{equation*}
Passing to the limit yields $\limsup_{n\to\infty}\frac{\#I_n}{\sqrt{\tilde k_n}}\leq C$ and we have $\#I_n=\mc{O}(\sqrt{\tilde k_n})$.\\
Now let $i\notin I_n$. By using the definition of $J_0$ and $J_0(z)\geq (J_0^{**})'(\ell)(z-\ell)+J_0^{**}(\ell)$, we deduce from $0\leq\sigma_n^i(\ell)\leq \frac{1}{\sqrt{\tilde k_n}}$
\begin{align}
 0\leq&J_2\left(\frac{u_n^{i+2}-u_n^i}{2\lambda_n}\right)+\frac{1}{2}J_1\left(\frac{u_n^{i+2}-u_n^{i+1}}{\lambda_n}\right)+\frac{1}{2}J_1\left(\frac{u_n^{i+1}-u_n^{i}}{\lambda_n}\right)\notag\\
&-J_0\left(\frac{u_n^{i+2}-u_n^i}{2\lambda_n}\right)\leq\frac{1}{\sqrt{\tilde k_n}},\label{ineq:lemma1}\\
 0\leq& J_0\left(\frac{u_n^{i+2}-u_n^i}{2\lambda_n}\right)-J_0^{**}(\ell)-(J_0^{**})'(\ell)\l(\frac{u_n^{i+2}-u_n^i}{2\lambda_n}-\ell\r)\leq\frac{1}{\sqrt{\tilde k_n}}\label{ineq:lemma2}.
\end{align}
Let $(h_n)\subset\N$ be such that $h_n\in\{0,...,k_n^1-1,k_n^2-1,...,n-2\}$ and $h_n\notin I_n$. By using the fact that $J_0(z)=J_0^{**}(\ell)+(J_0^{**})'(\ell)(z-\ell)$ if and only if $z=\min\{\ell,\gamma\}$, and [LJ3] we conclude from (\ref{ass:kn}) and (\ref{ineq:lemma2})
\[\frac{u_n^{h_n+2}-u_n^{h_n}}{2\lambda_n}\to\min\{\ell,\gamma\}\quad\mbox{as $n\to\infty$.}\] 
Combining this with (\ref{ineq:lemma1}) and assumption [LJ2], [LJ3], we deduce
 \begin{equation*}
  \lim_{n\to\infty}\frac{u_n^{h_n+1}-u_n^{h_n}}{\lambda_n}=\min\{\ell,\gamma\} \quad\mbox{ and }\quad \lim_{n\to\infty} \frac{u_n^{h_n+2}-u_n^{h_n+1}}{\lambda_n}=\min\{\ell,\gamma\}.\label{eq:lim1}
\end{equation*}
Hence, for sequences $(h_n^1),(h_n^2)\subset\N$ with $h_n^1\in\{0,...,k_n^1-3\}=:K_n^1$ and $h_n^2\in\{k_n^2+3,...,n-1\}=:K_n^2$ and $h_n^i\notin I_n$, for $n$ big enough and $i=1,2$, we deduce 
$$\lim_{n\to\infty} \frac{u_n^{h_n^i+1}-u_n^{h_n^i}}{\lambda_n}=\min\{\ell,\gamma\}.$$
It is left to prove existence of such sequences. Since $\#I_n=\mathcal O(\sqrt{\tilde k_n})$, we conclude by the definition of $k_n$ in (\ref{ass:kn}) that $K_n^i\setminus\l(I_n\cap K_n^i\r)\neq\emptyset$ for $n$ sufficiently large and $i=1,2$ which shows the existence.
\end{proof}

\subsection{The case $\ell\leq\gamma$}

Like in \cite{NNN}, we distinguish between the cases $\ell\leq\gamma$ and $\ell>\gamma$, where $\ell$ denotes the boundary condition on the last atom in the chain and $\gamma$ denotes the unique minimum point of $J_0$. In the case of $\ell\leq\gamma$ no fracture occurs by Proposition~\ref{prop:compact}. In this section, we show that the first order $\Gamma$-limits of $\H_n^{\ell,k_n,\T_n}$ and $H_{n}^{\ell}$ coincide if $\ell\leq\gamma$.\\
For any $0<\ell\leq\gamma$  and $\theta>0$, we define the boundary layer energy $B(\theta,\ell)$ as
\begin{equation}\label{eq:elasticboundary}
 \begin{split}
  B(\theta,\ell)=&\inf_{N\in\mathbb{N}}\min\bigg\{\frac{1}{2}J_1(v^1-v^0)+\sum_{i\geq0}\bigg\{J_2\left(\frac{v^{i+2}-v^i}{2}\right)+\frac{1}{2}J_1(v^{i+2}-v^{i+1})\\
 &\quad +\frac{1}{2} J_1(v^{i+1}-v^i)-J_0^{**}(\ell)-(J_0^{**})'(\ell)\left(\frac{v^{i+2}-v^i}{2}-\ell\right)\bigg\}:\\
 &\quad v:\mathbb{N}\rightarrow\mathbb{R},v^0=0,v^1=\theta,v^{i+1}-v^i=\ell \text{ if }i\geq N\bigg\}.
 \end{split}
\end{equation}
This was already defined in \cite{NNN}. The constraint on the difference $v^1-v^0$ is due to the boundary condition on the first and second atom and the  last and last but one. The terms in the sum have the same structure as $\sigma_n^i(\ell)$ defined in (\ref{def:sigmal}) and are always non-negative. 

\begin{theorem}\label{theorem:elastic}
 Suppose that hypotheses [LJ1]--[LJ4] hold. Let $0<\ell\leq\gamma$ and $u_0^{(1)},u_1^{(1)}>0$. Let $k_n^1,k_n^2$ satisfy (\ref{ass:kn}) and let $\T_n\subset \{0,1,...,n\}$ such that $\{0,1,...,k_n^1,k_n^2,...,n\}\subset \T_n$. Then $H_{1,n}^\ell$ as well as $\H_{1,n}^{\ell,k_n,\T_n}$ defined in (\ref{def:hn1qc}) $\Gamma$-converge with respect to the $L^\infty(0,1)$--topology to the functional $H_1^\ell$ defined by
\begin{equation*}
 H_1^\ell(u)=\begin{cases}
           B(u_0^{(1)},\ell)+B(u_1^{(1)},\ell)-J_0(\ell)-J_0'(\ell)\left(\frac{u_0^{(1)}+u_1^{(1)}}{2}-\ell\right) & \text{if } u(t)=\ell t,\\
	   +\infty & else
          \end{cases}
\end{equation*}
on $W^{1,\infty}(0,1)$.
\end{theorem}

\begin{proof} The proof for the convergence of $H_{1,n}^\ell$ is given in \cite[Theorem 4.1]{NNN}. Next we outline how this proof can be extended to the case $\H_{1,n}^{\ell,k_n,\T_n}$.\\

\textbf{Liminf inequality.} %It is sufficient to consider $\mathcal{T}_n= \{0,...,n\}$ since in other cases we have less sequences with finite energy. 
We show that for any sequence $u_n\rightarrow u$ in $L^\infty(0,1)$ with equibounded energy $\H_{1,n}^{\ell,k_n,\T_n}$ 
\begin{equation}
 \liminf_{n\to\infty} \H_{1,n}^{\ell,k_n,\T_n}(u_n)\geq B(u_0^{(1)},\ell)+B(u_1^{(1)},\ell)-J_0(\ell)-J_0'(\ell)\left(\frac{u_0^{(1)}+u_1^{(1)}}{2}-\ell\right).\label{eq:liminfelastic2}
\end{equation}
Proposition~\ref{prop:compact} implies that $u(t)=\ell t$ a.e. in $[0,1]$ and by Lemma~\ref{lemma:kngegen0} we can choose sequences of natural numbers $(h^1_n),(h^2_n)$ such that $h_n^1<k_n^1-2$, $h_n^2>k_n^2$ and
\begin{equation}\label{eq:elastic_h_n}
 \lim_{n\to\infty} \frac{u_n^{h^1_n+2}-u_n^{h^1_n+1}}{\lambda_n}=\ell,\quad\lim_{n\to\infty} \frac{u_n^{h^2_n+2}-u_n^{h^2_n+1}}{\lambda_n}=\ell.
\end{equation}
Using $\sigma_n^i(\ell),\mu_n^i(\ell)\geq0$, we obtain from (\ref{def:hnl1<})
\begin{align*}
 \H_{1,n}^{\ell,k_n,\T_n}(u_n)\geq&\frac12J_1(u_0^{(1)})+\sum_{i=0}^{h_n^1}\sigma_n^i(\ell)+\sum_{i=h_n^2+1}^{n-2}\sigma_n^i(\ell)+\frac12J_1(u_1^{(1)})-J_0^{**}(\ell)\\
&-\l(J_0^{**}\r)'(\ell)\l(\frac{u_0^{(1)}+u_1^{(1)}}{2}-\ell\r).
\end{align*}
By using (\ref{eq:elastic_h_n}) and the estimates \cite[(4.20)]{NNN} and \cite[(4.23)]{NNN}, we obtain
\begin{align}
 \frac12J_1\l(u_0^{(1)}\r)+\sum_{i=0}^{h_n^1}\sigma_n^i(\ell)\geq& B(u_0^{(1)},\ell)-\omega_1(n),\label{ineq:elasticinfleft}\\
 \frac12J_1\l(u_1^{(1)}\r)+\sum_{i=h_n^2+1}^{n-2}\sigma_n^i(\ell)\geq& B(u_1^{(1)},\ell)-\omega_2(n),\label{ineq:elasticinfright}
\end{align}
with $\omega_1(n),\omega_2(n)\to0$ as $n\to\infty$, which yields (\ref{eq:liminfelastic2}).\\

\textbf{Limsup inequality.} We can use the same recovery sequence as in the proof of \cite[Theorem 4.1]{NNN}. Since $H_{1}^{\ell}(u)$ is only finite if $u(t)=\ell t$ it is sufficient to consider just this case. We construct a sequence $(u_n)$ which satisfies the boundary conditions and converges to $u$ in $L^\infty(0,1)$ such that
\[\limsup_{n\to\infty}\H_{1,n}^{\ell,k_n,\T_n}(u_n)\leq B(u_0^{(1)},\ell)+B(u_1^{(1)},\ell)-J_0^{**}(\ell)-\l(J_0^{**}\r)'(\ell)\left(\frac{u_0^{(1)}+u_1^{(1)}}{2}-\ell\right).\]
Let $\eta>0$. By the definition of $B(u_0^{(1)},\ell)$, there exists $v:\mathbb{N}\rightarrow \mathbb{R}$ and $N_1\in\mathbb{N}$ such that $v^0=0,v^1=u_0^{(1)},v^{i+1}-v^i=\ell$ for $i\geq N_1$ and
\begin{equation}
\begin{split}
 \frac{1}{2}J_1(v^1-v^0)+\sum_{i\geq 0}\bigg\{J_2\left(\frac{v^{i+2}-v^{i}}{2}\right)+\frac{1}{2}\left(J_1\left(v^{i+2}-v^{i+1}\right)+J_1\left(v^{i+1}-v^{i}\right)\right)\\
-J_0^{**}(\ell)-\l(J_0^{**}\r)'(\ell)\left(\frac{v^{i+2}-v^i}{2}-\ell\right)\bigg\}\leq  B(u_0^{(1)},\ell)+\eta. \label{ineq:supleftboundary}
\end{split}
\end{equation}
Similarly we can find $w:-\mathbb{N}\to\mathbb{R}$ and $N_2\in\mathbb{N}$ with $w^0=0,w^0-w^{-1}=u_1^{(1)},w^i-w^{i-1}=\ell$ if $i\leq-N_2$ such that
\begin{equation}
\begin{split}
\frac{1}{2}J_1(w^0-w^{-1})+\sum_{i\leq 0}\bigg\{J_2\left(\frac{w^{i}-w^{i-2}}{2}\right)+\frac{1}{2}\left(J_1\left(w^{i}-w^{i-1}\right)+J_1\left(w^{i-1}-w^{i-2}\right)\right)\\
-J_0^{**}(\ell)-\l(J_0^{**}\r)'(\ell)\left(\frac{w^{i}-w^{i-2}}{2}-\ell\right)\bigg\}\leq  B(u_1^{(1)},\ell)+\eta.\label{ineq:suprightboundary} 
\end{split}
\end{equation}
By means of the functions $v$ and $w$ we can construct a recovery sequence $(u_n)$ for $u$ 
$$
 u_n^i=\begin{cases}
        \lambda_n v^i & \text{if } 0\leq i\leq N_1+2,\\
	\lambda_nv^{N_1+2}+\frac{\ell+\lambda_n(w^{-N_2-2}-v^{N_1+2})}{n-N_1-N_2-4}(i-N_1-2) &\text{if } N_1+2\leq i\leq n-N_2-2,\\
	\ell+\lambda_nw^{i-n}&\text{if }n-N_2-2\leq i\leq n.\\
       \end{cases}
$$
The functions $v$ and $w$ are chosen in such a way that $u_n$ satisfies the boundary conditions (\ref{def:boundarycond}) for every $n\in\mathbb{N}$. Moreover, since $k_n^1\to+\infty$ and $n-k_n^2\to+\infty$ we can assume $N_1+2\leq k_n^1$ and $n-N_2-2\geq k_n^2$. This implies that $u_n$ is linear on $\lambda_n(k_n^1,k_n^2)$ and thus $u_n\in\A_{\T_n}(0,1)$ for arbitrary $\T_n$ satisfying $\{0,...,k_n^1,k_n^2,..n\}\subset\T_n$. Using (\ref{ineq:supleftboundary}) and (\ref{ineq:suprightboundary}) we obtain
\begin{align*}
 \frac{1}{2}J_1\left(\frac{u_n^1-u_n^0}{\lambda_n}\right)+\sum_{i=0}^{N_1}\sigma_n^i(\ell)\leq& B(u_0^{(1)},\ell)+\eta,\\
 \frac{1}{2}J_1\left(\frac{u_n^n-u_n^{n-1}}{\lambda_n}\right)+\sum_{i=n-N_2-2}^{n-2}\sigma_n^i(\ell)\leq& B(u_1^{(1)},\ell)+\eta,
\end{align*}
which is shown in detail in \cite{NNN}. It remains to show that
$$\Sigma:=\sum_{i=N_1+1}^{k_n^1-1}\sigma_n^i(\ell)+\frac{1}{2}\mu_n^{k^1_n}(\ell)+\sum_{i=k_n^1+1}^{k_n^2-2}\mu_n^i(\ell)+\frac{1}{2}\mu_n^{k_n^2-1}(\ell)+\sum_{i=k_n^2-1}^{n-N_2-3}\sigma_n^i(\ell) $$ 
is infinitesimal as $n\to\infty$. This follows also directly from the proof of \cite[Theorem 4.1]{NNN}. Indeed, in \cite[Theorem 4.1]{NNN} it is shown that for the above sequence it holds $\sum_{i=N_1+1}^{n-N_2-3}\sigma_n^i(\ell)$ tends to zero as $n\to\infty$. By using the fact that $u_n$ is linear on $\lambda_n(N_1+2,n-N_2-2)$ we have $\sigma_n^i(\ell)=\mu_n^i(\ell)$ for $i=N_1+2,...,n-N_2-4$ and thus the statement follows.
\end{proof}

\begin{remark}\label{rem:elastic}
%  (i) Theorem~\ref{theorem:elastic} shows that the functionals $H_n^\ell$ and $H_n^{\ell,k_n,\T_n}$ are $\Gamma$-equivalent of order $\lambda_n$, see \cite[Definition 4.2]{BT} for the definition. Thus, by \cite[Theorem 4.4]{BT}
% \[\min_u H_n^{\ell}=\min_u \H_n^{\ell,k_n,\T_n}+o(\lambda_n).\]
 In the proof of Theorem~\ref{theorem:elastic}, the assumption (\ref{ass:kn}) (i) is crucial. If one drops this assumption, for example to let $k_n^1$ and $n-k_n^2$ be independent of $n$, the first order $\Gamma$-limits of $H_n^{\ell,k_n,\T_n}$ and $\H_n^\ell$ do not coincide in general. In this case the boundary layer energies $B(\theta,\ell)$ would be replaced by some ``truncated'' boundary layer energies $\tilde B(\theta,\ell)$ in the first order $\Gamma$-limit of $\H_n^{\ell,k_n\T_n}$. To quantify the difference between $B(\theta,\ell)$ and $\tilde B(\theta,\ell)$ one has to perform a deeper analysis, as in \cite{Hud}, on the decay of the boundary layers.   
\end{remark}

%
% Case \ell>\gamma
%
\subsection{The case $\ell>\gamma$}
According to Proposition~\ref{prop:compact}, the case $\ell>\gamma$ leads to fracture. Each crack costs a certain amount of fracture energy, cf.\ \cite{BCi,NNN}. We will show that this fracture energy depends on whether the crack is located in $(0,1)$ or $\{0,1\}$ and on the choice of the representative atoms $\T=(\T_n)$ close to the crack.\\
We repeat the definition of the boundary layer energy when fracture occurs at a boundary point from \cite{NNN}. For $\theta>0$, this is given by
\begin{equation}\label{fracture:bb}
 \begin{split}
  B_b(\theta)=&\inf_{k\in\mathbb{N}}\min\bigg\{\frac{1}{2}J_1(v^1-v^0)+\sum_{i=0}^{k-1}\bigg\{J_2\left(\frac{v^{i+2}-v^i}{2}\right)\\
 &\quad +\frac{1}{2}J_1(v^{i+2}-v^{i+1})+\frac{1}{2} J_1(v^{i+1}-v^i)-J_0(\gamma)\bigg\}:\\
 &\quad v:\mathbb{N}\rightarrow\mathbb{R},v^{k+1}=0,v^{k+1}-v^k=\theta\bigg\}.
 \end{split}
\end{equation}
We define $B(\gamma)$ as in \cite{BCi,NNN}
\begin{equation}\label{fracture:b}
 \begin{split}
  B(\gamma)=&\inf_{N\in\mathbb{N}}\min\bigg\{\frac{1}{2}J_1(v^1-v^0)+\sum_{i\geq0}\bigg\{J_2\left(\frac{v^{i+2}-v^i}{2}\right)\\
 &\quad +\frac{1}{2}J_1(v^{i+2}-v^{i+1})+\frac{1}{2} J_1(v^{i+1}-v^i)-J_0(\gamma)\bigg\}:\\
 &\quad v:\mathbb{N}\rightarrow\mathbb{R},v^0=0,v^{i+1}-v^i=\gamma\text{ if }i\geq N\bigg\}.
 \end{split}
\end{equation}

Next we recall \cite[Theorem 4.2]{NNN} and explain how this theorem changes in the case of the above quasicontinuum model.
\begin{theorem}\cite[Theorem 4.2.]{NNN}\label{th:fracturennn}
 Suppose that hypotheses [LJ1]--[LJ4] hold. Let $\ell>\gamma$ and $u_0^{(1)},u_1^{(1)}>0$. Then $H_{1,n}^{\ell}$ $\Gamma$-converges with respect to the $L^1(0,1)$--topology to the functional $H_{1}^{\ell}$ defined by
\begin{equation}\label{def:limitfrac}
\begin{split}
H_{1}^{\ell}(u)=&B\l(u_0^{(1)},\gamma\r)(1-\#(S_u\cap\{0\}))+B\l(u_1^{(1)},\gamma\r)(1-\#(S_u\cap\{1\}))-J_0(\gamma)\\
	  &+B_{BJ}\l(u_0^{(1)}\r)\#(S_u\cap\{0\})+B_{BJ}\l(u_1^{(1)}\r)\#(S_u\cap\{1\})+B_{IJ}\#\l(S_u\cap\l(0,1\r)\r)
\end{split}
\end{equation}
if $u\in SBV_c^\ell(0,1)$, and $+\infty$ else on $L^1(0,1)$, where, for $\theta>0$,
\begin{align}
 B_{BJ}(\theta)=\frac{1}{2}J_1(\theta)+B_b(\theta)+B(\gamma)-2J_0(\gamma)\label{eq:fractureBJ}
\end{align}
is the boundary layer energy due to a jump at the boundary, while
\begin{align}
 B_{IJ}=2B(\gamma)-2J_0(\gamma)\label{eq:fractureIJ}
\end{align}
is the boundary layer energy due to a jump in an internal point of $(0,1)$ and $B(\theta,\gamma)$ denotes the elastic boundary layer energy defined in (\ref{eq:elasticboundary}).
\end{theorem}

We aim for an analogous result for $\H_{1,n}^{\ell,k_n,\T_n}$. Here the specific structure of $\T=(\T_n)$ turns out to be important. We will show that every jump corresponds to the debonding of a pair of representative atoms and this induces the debonding of all atoms in between. Thus the distance between two neighbouring repatoms quantifies the jump energy. For given $k_n=(k_n^1,k_n^2)$, $x\in[0,1]$, we assume that $\T=(\T_n)$ is such that the following limit exists in $\N\cup\{+\infty\}$%$\H_{1,n}^{\ell,k_n,\T_n}$  
\begin{equation}\label{def:B(x,T)}
 \begin{split}
b(x,\T):=&\lim_{n\to\infty}\min\big\{q_n^2-q_n^1\, :\, (q_n^1),(q_n^2)\subset\N,k_n^1<q_n^1<q_n^2<k_n^2,\\
 &\hspace*{2.0cm}~q_n^1,~q_n^2\in\T_n,\lim_{n\to\infty}\lambda_nq_n^1=\lim_{n\to\infty}\lambda_nq_n^2=x\big\}.
 \end{split}
\end{equation}
The choice of repatoms at the interface between the local and nonlocal region has to be treated with extra care and we assume that the following limits exist in $\N\cup\{+\infty\}$
\begin{equation}\label{def:rl}
\begin{split}
 \hat{r}(\T)&:=\lim_{n\to\infty}\l(r(\T_n)-k_n^1\r),\mbox{ with }r(\T_n):=\min\{r\in\T_n:k_n^1<r\}, \\
 \hat{l}(\T)&:=\lim_{n\to\infty}\l(k_n^2-l(\T_n)\r),\mbox{ with }l(\T_n):=\max\{l\in\T_n:k_n^2>l\}.
\end{split}
\end{equation} 
 Moreover, we define for $m\in\mathbb{N}$ the following minimum problem
\begin{equation}\label{fracture:bif}
\begin{split}
 B_{IF}(m)=&\inf_{k\in\mathbb{N}}\min\bigg\{\frac{1}{2}J_1(v^1-v^0)+\sum_{i=0}^{k-1}\bigg\{J_2\left(\frac{v^{i+2}-v^i}{2}\right)\\
  &\quad +\frac{1}{2}J_1(v^{i+2}-v^{i+1})+\frac{1}{2} J_1(v^{i+1}-v^i)-J_0(\gamma)\bigg\}\\
  &\quad+\frac{2m+1}{2}\left(J_{CB}(v^{k+1}-v^k)-J_0(\gamma)\right):v:\mathbb{N}\rightarrow\mathbb{R},v^{0}=0\bigg\},
\end{split}
\end{equation}
which corresponds to a jump in the atomistic region at the atomistic/continuum interface, where $m$ corresponds to the distance between the neighbouring repatoms at the interface, specified below. Furthermore, we set $B_{IF}(\infty)=B(\gamma)$.  
\begin{lemma}\label{lemma:repcomp}
Let $J_1,J_2$ be potentials such that [LJ1]--[LJ4] hold. Let $\T_n=\{t_n^0,t_n^1,...,t_n^{r_n}\}$ with $0=t_n^0<t_n^1<...<t_n^{r_n}=n$ for all $n\in\N$. Let $(u_n)$ be a sequence of functions satisfying (\ref{eq:prop}). Furthermore, let $(h_n)\subset\N$ be such that $k_n^1\leq t_n^{h_n}<t_n^{h_n+1}\leq k_n^2$ and $\liminf_{n\to\infty}\l(t_n^{h_n+1}-t_n^{h_n}\r)=+\infty$. Then, we have 
\[\lim_{n\to\infty}\l(\frac{u_n^{t_n^{h_n}+1}-u_n^{t_n^{h_n}}}{\lambda_n}\r)=\gamma.\]
\end{lemma}
\begin{proof}
From the equiboundedness of $\sup_n\H_{1,n}^{\ell,k_n,\T_n}(u_n)$, we deduce the existence of a constant $C>0$ such that
$$C\geq\sup_{n}\sum_{i=t_n^{h_n}}^{t_n^{h_n+1}-1}\mu_n^i=\sup_n(t_n^{h_n+1}-t_n^{h_n})\mu_n^{t_n^{h_n}},$$
where we used the fact that $u_n'(x)=\lambda_n^{-1}(t_n^{h_n+1}-t_n^{h_n})^{-1}(u_n^{t_n^{h_n+1}}-u_n^{t_n^{h_n}})$ for all $x\in\lambda_n(t_n^{h_n},t_n^{h_n+1})$. This implies $\mu_n^{t^{h_n}}=\mc{O}((t_n^{h_n+1}-t_n^{h_n})^{-1})$ and thus  $\mu_n^{t^{h_n}}\to0$ as $n\to\infty$. Similar steps as in Lemma~\ref{lemma:kngegen0} now lead to 
\[\lim_{n\to\infty}\l(\frac{u_n^{t_n^{h_n}+1}-u_n^{t_n^{h_n}}}{\lambda_n}\r)=\gamma.\]
\end{proof}

Next, we will state the main theorem of this section concerning the $\Gamma$-limit of the functionals $\H_{1,n}^{\ell,k_n,\T_n}$ for $\ell>\gamma$. The $\Gamma$-limit is different to the one obtained for $H_1^\ell$ in \cite{NNN}, cf.\ Theorem~\ref{th:fracturennn}. We will come back to this in section 5.
\begin{theorem}\label{theorem:fracture}
 Suppose that hypotheses [LJ1]--[LJ4] hold. Let $\ell>\gamma$ and $u_0^{(1)}$, $u_1^{(1)}>0$. Let $(k_n^1),(k_n^2)$ satisfy (\ref{ass:kn}) and let $\T=(\T_n)$ satisfy (\ref{T_n}) such that
\begin{equation}
\{0,...,k_n^1\}\cup\{k_n^2,...,n\}\subset \T_n=\{t_n^0,....,t_n^{r_n}\}\label{def:atoms}
\end{equation}
and the limits defined in (\ref{def:B(x,T)}) and (\ref{def:rl}) exist in $\N\cup\{+\infty\}$. Then $\H_{1,n}^{\ell,k_n,\T_n}$ defined in (\ref{def:hn1qc}) $\Gamma$-converges with respect to the $L^1(0,1)$--topology to the functional $\H_1^{\ell,\T}$ defined by
\begin{align}\label{def:limitfracqc}
 \H_1^{\ell,\T}(u)=&B\l(u_0^{(1)},\gamma\r)(1-\#(S_u\cap\{0\}))+B\l(u_1^{(1)},\gamma\r)(1-\#(S_u\cap\{1\}))\notag\\
  &+B_{IFJ}\l(\hat{r}(\T),b(0,\T),u_0^{(1)}\r)\#\l(S_u\cap\{0\}\r)-\sum_{x:x\in S_u\cap(0,1)}b(x,\T)J_0(\gamma)\notag\\
  &+B_{IFJ}\l(\hat{l}(\T),b(1,\T),u_1^{(1)}\r)\#\l(S_u\cap\{1\}\r)-J_0(\gamma)
\end{align}
if $u\in SBV_c^\ell(0,1)$, and $+\infty$ else on $L^1(0,1)$, where $B_{IFJ}(n,k,\theta)$ is defined for $n,k\in\N\cup\{+\infty\}$, $\theta>0$ as  
\begin{align}\label{def:bifjxint}
%\begin{split}
 B_{IFJ}(n,k,\theta)=&\min\bigg\{\min\left\{B_{AIF}(n),B(\gamma)-\l(\frac12+n\r)J_0(\gamma),-kJ_0(\gamma)\right\}+B(\theta,\gamma),\notag\\
&\hspace*{1.2cm}B_{BJ}(\theta)\bigg\}
%\end{split}
\end{align}
with
\begin{equation}\label{def:baif}
 B_{AIF}(n)=B_{IF}(n-1)+B(\gamma)-2J_0(\gamma),
\end{equation}
where $B_{BJ}$ and $B_{IF}$ are given in (\ref{eq:fractureBJ}) and (\ref{fracture:bif}).
\end{theorem}

\begin{remark}\label{rem:theoremfracture}
 In \cite{NNN} it is shown that $B_{BJ}(\theta)$ and $B_{IJ}$ are positive. The same holds true for $B_{IFJ}(n,k,\theta)$, see Lemma~\ref{lemma:bifj}. Hence all jump energies are positive.  
\end{remark}

\begin{proof} 
\textbf{\underline{Liminf inequality.}} Since the jump energies are positive (Remark~\ref{rem:theoremfracture}) we can assume without loss of generality that there is only one jump point. By symmetry, we only need to distinguish between a jump in $0$ and in $(0,1)$. 
  
\textbf{Jump in $0$.} Let $(u_n)$ be a sequence of functions converging to $u$ with $S_u=\{0\}$ such that $\sup_n \H_{1,n}^{\ell,k_n,\T_n} (u_n)<+\infty$. Then Proposition~\ref{prop:compact} implies that $u_n\to u$ in $L^1(0,1)$ with
\begin{equation}
 u(t)=\begin{cases}
       0 &\text{if } t=0,\\
       (\ell-\gamma)+\gamma t &\text{if } 0<t\leq1.
      \end{cases}\label{ucrackxi}
\end{equation}
By Lemma~\ref{lemma:kngegen0} there exist sequences $(T_n^1),(T_n^2)\subset\N$ with $0<T_n^1<k_n^1-1<k_n^2+1<T_n^2<n-2$ such that
\begin{equation}
 \lim_{n\to\infty}\frac{u_n^{T_n^1+2}-u_n^{T_n^1+1}}{\lambda_n}=\gamma,\quad\lim_{n\to\infty}\frac{u_n^{T_n^2+2}-u_n^{T_n^2+1}}{\lambda_n}=\gamma.\label{eq:fracT_n}
\end{equation}
We can write the energy in (\ref{def:hnl1<}) as
\begin{equation}
 \begin{split}
 \H_{1,n}^{\ell,k_n,\T_n}(u_n)=&\frac{1}{2}J_1\left(\frac{u_n^1-u_n^0}{\lambda_n}\right)+\sum_{i=0}^{T_n^1}\sigma_n^i+\sum_{i=T_n^1+1}^{k^1_n-1}\sigma_n^i+\frac{1}{2}\mu_n^{k_n^1}+\sum_{i=k_n^1+1}^{k^2_n-2}\mu_n^i\\
& +\frac12\mu_n^{k_n^2-1}+\sum_{i=k^2_n-1}^{T_n^2}\sigma_n^i+\sum_{i=T_n^2+1}^{n-2}\sigma_n^i+\frac12J_1\l(\frac{u_n^n-u_n^{n-1}}{\lambda_n}\r)-J_0(\gamma).\label{liminfdecomp}  
 \end{split}
\end{equation}
The estimate for the elastic boundary layer energy at $1$ is exactly the same as in the case $\ell\leq\gamma$, see (\ref{ineq:elasticinfright}), and is given by
%: Following the corresponding proof in Theorem~\ref{theorem:elastic} and replace $h_n^2$ by $T_n^2$ and $\ell$ by $\gamma$. Hence we have
\begin{equation}
 \liminf_{n\to\infty}\l(\sum_{i=T_n^2+1}^{n-2}\sigma_n^i+\frac12J_1\l(u_1^{(1)}\r)\r)\geq B(u_1^{(1)},\gamma).\label{frac:rightboundary}%\\
\end{equation}
To estimate the remaining terms, we note that there exists $(h_n)\subset\N$ with $\lambda_nh_n\to0$ such that
\begin{equation}
 \lim_{n\to\infty}\frac{u_n^{h_n+1}-u_n^{h_n}}{\lambda_n}=+\infty,\label{lim_hninf}
\end{equation}
as argued in the proof of \cite[Theorem 4.2]{NNN}. Here we have to consider the following cases:
\begin{equation}
 (1)~ h_n\leq T_n^1,\quad (2)~T_n^1+1<h_n<k_n^1, \quad (3)~k_n^1\leq h_n<r(\T_n),\quad (4)~r(\T_n)\leq h_n.\label{cases}
\end{equation}
Indeed, it is enough to consider the above cases. By extracting a subsequence, we can assume that $\liminf_{n\to\infty} \H_{1,n}^{\ell,k_n,\T_n}(u_n)=\lim_{n\to\infty}\H_{1,n}^{\ell,k_n,\T_n}(u_n)$. Let $(h_n)$ be such that it oscillates between at least two of the cases (1)--(4), then we can extract a further subsequence which satisfies only one of the cases, which does not change the limit.\\
The first two cases correspond to a jump in the atomistic region. In the first case, the jump is sufficiently far from the atomistic/continuum interface and leads to the same jump energy as a jump in $0$ in the fully atomistic model. The jump in the second case is closer to the continuum region and leads to a jump energy of the form $B_{AIF}(n)$, see (\ref{def:baif}). In the third case, the jump is exactly at the interface between the atomistic region and the continuum region. The last case corresponds to a jump within the continuum region.

\textit{\underline{Case (1):}} Consider $(u_n)$ as above with $(h_n)$ satisfying (\ref{lim_hninf}) and (\ref{cases}, (1)). We show that
\begin{equation}
 \liminf_{n\to\infty} \H_{1,n}^{\ell,k_n,\T_n}(u_n)\geq B(u_1^{(1)},\gamma)+\frac12J_1(u_0^{(1)})+B_b(u_0^{(1)})+B(\gamma)-3J_0(\gamma).\label{ineq:infcrack0}
\end{equation}
This can be proven in the same way as the corresponding inequality for a jump in $0$ in \cite[Theorem 4.2]{NNN}. By (\ref{liminfdecomp}) and (\ref{frac:rightboundary}), we only need to estimate 
\begin{align*}
 \sum_{i=0}^{h_n-2}\sigma_n^i+\sigma_n^{h_n-1}+\sigma_n^{h_n}+\sum_{i=h_n+1}^{T^1_n}\sigma_n^i=&\frac{1}{2}J_1\left(\frac{u_n^{h_n}-u_n^{h_n-1}}{\lambda_n}\right)+\sum_{i=0}^{h_n-2}\sigma_n^i+\sum_{i=h_n+1}^{T^1_n}\sigma_n^i\\
&+\frac{1}{2}J_1\left(\frac{u_n^{h_n+2}-u_n^{h_n+1}}{\lambda_n}\right)-2J_0(\gamma)+\omega(n),
\end{align*}
with
$$
 \omega(n)=J_2\left(\frac{u_n^{h_n+1}-u_n^{h_n-1}}{2\lambda_n}\right)+J_1\left(\frac{u_n^{h_n+1}-u_n^{h_n}}{\lambda_n}\right)+J_2\left(\frac{u_n^{h_n+2}-u_n^{h_n}}{2\lambda_n}\right),
$$
which converges to $0$ as $n\to\infty$, since $J_1(\infty)=J_2(\infty)=0$. As shown in \cite[(4.39)]{NNN} and \cite[(4.40)]{NNN} it holds
\begin{align}
 \sum_{i=0}^{h_n-2}\sigma_n^i+\frac12J_1\l(\frac{u_n^{h_n}-u_n^{h_n-1}}{\lambda_n}\r)\geq& B_b(u_0^{(1)}),\label{ineq:boundcrack}\\
 \frac{1}{2}J_1\left(\frac{u_n^{h_n+2}-u_n^{h_n+1}}{\lambda_n}\right)+\sum_{i=h_n+1}^{T^1_n}\sigma_n^i\geq& B(\gamma)+r_2(n),\label{ineq:infcracka}
\end{align}
with $\lim_{n\to\infty}r_2(n)=0$. By using (\ref{frac:rightboundary}), (\ref{ineq:boundcrack}), (\ref{ineq:infcracka}) and the fact that $\sigma_n^i,\mu_n^i\geq0$, we obtain (\ref{ineq:infcrack0}).

\textit{\underline{Case (2):}} Assume that $(u_n)$ satisfies (\ref{lim_hninf}) with $(h_n)$ such that (\ref{cases}, (2)) holds true. We show that  
\begin{equation}
 \liminf_{n\to\infty} \H_{1,n}^{\ell,k_n,\T_n}(u_n)\geq B(u_0^{(1)},\gamma)+B(u_1^{(1)},\gamma)+B(\gamma)+B_{IF}(\hat{r}(\T)-1)-3J_0(\gamma).\label{ineq:infcrack1}
\end{equation}
First of all we estimate the elastic boundary layer energy at $0$ as in the case $\ell\leq\gamma$, see (\ref{ineq:elasticinfleft}), and obtain
\begin{equation}
 \liminf_{n\to\infty}\l(\frac12J_1\l(u_0^{(1)}\r)+\sum_{i=0}^{T_n^1}\sigma_n^i\r)\geq B(u_0^{(1)},\gamma).\label{frac:leftboundary}
\end{equation}
It remains to estimate 
\begin{align*}
 &\sum_{i=T_n^1+1}^{h_n-2}\sigma_n^i+\sigma_n^{h_n-1}+\sigma_n^{h_n}+\sum_{i=h_n+1}^{k^1_n-1}\sigma_n^i+\frac{1}{2}\mu_n^{k^1_n}+\sum_{i=k_n^1+1}^{k^2_n-2}\mu_n^i\\
 &\qquad =\frac{1}{2}J_1\left(\frac{u_n^{h_n}-u_n^{h_n-1}}{\lambda_n}\right)+\sum_{i=T_n^1+1}^{h_n-2}\sigma_n^i+\frac{1}{2}J_1\left(\frac{u_n^{h_n+2}-u_n^{h_n+1}}{\lambda_n}\right)+\sum_{i=h_n+1}^{k^1_n-1}\sigma_n^i\\
 &\qquad\quad-2J_0(\gamma)+\frac{1}{2}\mu_n^{k^1_n}+\sum_{i=k^1_n+1}^{k^2_n-2}\mu_n^i+\omega(n),
\end{align*}
with
$$
 \omega(n)=J_2\left(\frac{u_n^{h_n+1}-u_n^{h_n-1}}{2\lambda_n}\right)+J_1\left(\frac{u_n^{h_n+1}-u_n^{h_n}}{\lambda_n}\right)+J_2\left(\frac{u_n^{h_n+2}-u_n^{h_n}}{2\lambda_n}\right),
$$
which converges to $0$ as $n\to\infty$, since $J_1(\infty)=J_2(\infty)=0$. As in \cite[(4.48)]{NNN} we obtain
\begin{equation}
\frac{1}{2}J_1\left(\frac{u_n^{h_n}-u_n^{h_n-1}}{\lambda_n}\right)+\sum_{i=T_n^1+1}^{h_n-2}\sigma_n^i \geq B(\gamma)+r_1(n),\label{ineq:intcrack1}
\end{equation}
with $r_1(n)\to 0$ as $n\to\infty$. Next we show for $\hat{r}(\T)<\infty$ that
\begin{equation}
\liminf_{n\to\infty}\left\{\frac{1}{2}J_1\left(\frac{u_n^{h_n+2}-u_n^{h_n+1}}{\lambda_n}\right)+\sum_{i=h_n+1}^{k^1_n-1}\sigma_n^i+\frac{1}{2}\mu_n^{k^1_n}+\sum_{i=k_n^1+1}^{r(\T_n)-1}\mu_n^i\right\}\geq B_{IF}(\hat{r}(\T)-1).\label{ineq:intcrack2}
\end{equation}
To this end we define for $j=0,...,r(\T_n)-h_n$ 
\[\hat{u}_n^j=\frac{u_n^{h_n+1+j}-u_n^{h_n+1}}{\lambda_n}.\]
By definition of $\hat{r}(\T)$, see (\ref{def:rl}), there exists an $N\in\mathbb{N}$ such that $r(\T_n)-k_n^1 = \hat{r}(\T)$ for all $n\geq N$. From $u_n\in\A_{\T_n}(0,1)$ and (\ref{def:atoms}) we easily deduce $\mu_n^i=\mu_n^{k_n^1}$ for $k_n^1\leq i\leq r(\T_n)-1$. Hence
\begin{align*}
 &\frac{1}{2}J_1\left(\frac{u_n^{h_n+2}-u_n^{h_n+1}}{\lambda_n}\right)+\sum_{i=h_n+1}^{k^1_n-1}\sigma_n^i+\frac{1}{2}\mu_n^{k^1_n}+\sum_{i=k_n^1+1}^{r(\T_n)-1}\mu_n^i\\
 &\quad\geq\frac{1}{2}J_1(\hat{u}_n^1-\hat{u}_n^0)+\sum_{j=0}^{k^1_n-h_n-2}\bigg\{J_2\left(\frac{\hat{u}_n^{j+2}-\hat{u}_n^j}{2}\right)-J_0(\gamma)+\frac{1}{2}(J_1(\hat{u}_n^{j+2}-\hat{u}_n^{j+1})\\
 &\qquad+J_1(\hat{u}_n^{j+1}-\hat{u}_n^{j}))\bigg\}+\l(\frac{1}{2}+\hat{r}(\T)-1\r)\left(J_{CB}(\hat{u}_n^{k^1_n-h_n}-\hat{u}_n^{k^1_n-h_n-1})-J_0(\gamma)\right).
\end{align*}
Since $\hat{u}_n^0=0$, this is an admissible test for $ B_{IF}(\hat{r}(\T)-1)$ and (\ref{ineq:intcrack2}) holds true.\\
In case of $\hat{r}(\T)=\infty$, we deduce from Lemma~\ref{lemma:repcomp} that $\frac{u_n^{k_n^1+1}-u_n^{k_n^1}}{\lambda_n}\to\gamma$ as $n\to\infty$. Thus, we obtain as in (\ref{ineq:infcracka}) 
\begin{equation}
\frac{1}{2}J_1\left(\frac{u_n^{h_n+2}-u_n^{h_n+1}}{\lambda_n}\right)+\sum_{i=h_n+1}^{k_n^1-1}\sigma_n^i \geq B(\gamma)+r_1(n)=B_{IF}(\infty)+r_1(n),\label{ineq:intcrack3}
\end{equation}
with $r_1(n)\to0$ as $n\to+\infty$. By using (\ref{frac:rightboundary}), (\ref{frac:leftboundary})--(\ref{ineq:intcrack3}) and the fact that $\sigma_n^i,\mu_n^i\geq0$, we obtain (\ref{ineq:infcrack1}).

\textit{\underline{Case (3):}} Let $(u_n)$ satisfy (\ref{lim_hninf}) with $(h_n)$ such that (\ref{cases}) (3) holds true. We show  
\begin{equation}
 \liminf_{n\to\infty}\H_{1,n}^{\ell,k_n,\T_n}(u_n)\geq B(u_0^{(1)},\gamma)+B(u_1^{(1)},\gamma)+B(\gamma)-\l(\frac12+\hat{r}(\T)\r)J_0(\gamma).\label{ineq:infcrack2}
\end{equation}
Let $\hat r(\T)=+\infty$. By Lemma~\ref{lemma:repcomp}, we deduce $\lim_{n\to\infty}\frac{1}{\lambda_n}\l(u_n^{k_n^1+1}-u_n^{k_n^1}\r)=\gamma$ which is a contradiction to the existence of $(h_n)$ satisfying (\ref{lim_hninf}) and (\ref{cases}) (3). Hence, we can assume $\hat r(\T)<+\infty$. Next we estimate
\begin{align*}
 &\sum_{i=T_n^1+1}^{k_n^1-2}\sigma_n^i+\sigma_n^{k_n^1-1}+\frac{1}{2}\mu_n^{k^1_n}+\sum_{i=k_n^1+1}^{r(\T_n)-1}\mu_n^i+\sum_{i=r(\T_n)}^{k^2_n-2}\mu_n^i \\
 &=\frac{1}{2}J_1\left(\frac{u_n^{k_n^1}-u_n^{k_n^1-1}}{\lambda_n}\right)+\sum_{i=T_n^1+1}^{k_n^1-2}\sigma_n^i-\frac32J_0(\gamma)-(r(\T_n)-k_n^1-1)J_0(\gamma)\\
&\quad+\sum_{i=r(\T_n)}^{k^2_n-2}\mu_n^i+\omega(n),
\end{align*}
where
\begin{align*}
 \omega(n)=&\frac12J_1\l(\frac{u_n^{k_n^1+1}-u_n^{k_n^1}}{\lambda_n}\r)+J_2\l(\frac{u_n^{k_n^1+1}-u_n^{k_n^1-1}}{2\lambda_n}\r)\\
&+\l(r(\T_n)-k_n^1-\frac12\r)J_{CB}\l(\frac{u_n^{k_n^1+1}-u_n^{k_n^1}}{\lambda_n}\r)
\end{align*}
which converges to zero as $n$ tends to $+\infty$. Moreover, we obtain by \cite[(4.48)]{NNN} 
\begin{equation}
 \frac{1}{2}J_1\left(\frac{u_n^{k_n^1}-u_n^{k_n^1-1}}{\lambda_n}\right)+\sum_{i=T_n^1+1}^{k_n^1-2}\sigma_n^i\geq B(\gamma)+r(n),\label{ineq:intcrack4}
\end{equation}
with $\lim_{n\to\infty}r(n)=0$. Combining   (\ref{def:rl}), (\ref{frac:rightboundary}), (\ref{frac:leftboundary}), (\ref{ineq:intcrack4}) and the fact that $\mu_n^i\geq0$, we prove assertion (\ref{ineq:infcrack2}).

\textit{\underline{Case (4):}} Finally, let $(u_n)$ satisfy (\ref{lim_hninf}) with $(h_n)$ such that (\ref{cases}) (4) holds. We show  
\begin{equation}
 \liminf_{n\to\infty}\H_{1,n}^{\ell,k_n,\T_n}(u_n)\geq B(u_0^{(1)},\gamma)+B(u_1^{(1)},\gamma)-\l(b(0,\T)+1\r)J_0(\gamma).\label{ineq:infcrack3}
\end{equation}
With a similar argument as in case (3), we deduce from Lemma~\ref{lemma:repcomp} that $b(0,\T)$ has to be finite. There exists $(q_n)$ such that $t_n^{q_n}\leq h_n < t_n^{q_n+1}$ where $t_n^{q_n},~t_n^{q_n+1}\in\T_n$. For $u_n\in\A_{\T_n}(0,1)$, we have $\mu_n^i=\mu_n^{h_n}$ for $t_n^{q_n}\leq i\leq t_n^{q_n+1}-1$. By using $\mu_n^i,\sigma_n^i\geq0$, we obtain
$$
 C\geq\sum_{i=T_n^1+1}^{k_n^1-1}\sigma_n^i+\frac{1}{2}\mu_n^{k^1_n}+\sum_{i=k_n^1+1}^{t_n^{q_n}-1}\mu_n^i+\sum_{i=t_n^{q_n}}^{t_n^{q_n+1}-1}\mu_n^i+\sum_{i=t_n^{q_n+1}}^{k^2_n-2}\mu_n^i \geq(t_n^{q_n+1}-t_n^{q_n})\mu_n^{h_n}.
$$
Since $\mu_n^{h_n}\geq0$, $\lim_{n\to\infty}\mu_n^{h_n}=-J_0(\gamma)$ and since there exists, using (\ref{T_n}), a constant $N\in\N$ such that $(t_n^{q_n+1}-t_n^{q_n})\geq b(0,\T)$ for all $n\geq N$, we get 
$$
 \liminf_{n\to\infty}(t_n^{q_n+1}-t_n^{q_n})\mu_n^{h_n}\geq b(0,\T)\liminf_{n\to\infty}\mu_n^{h_n}=-b(0,\T)J_0(\gamma),
$$
which proves together with (\ref{frac:rightboundary}) and (\ref{frac:leftboundary}) inequality (\ref{ineq:infcrack3}).\\

In summary, for the jump in $0$, we have the estimate
\begin{align*}
\liminf_{n\to\infty} \H_{1,n}^{\ell,k_n,\T_n}(u_n)\geq& B(u_1^{(1)},\gamma)-J_0(\gamma)+\min\bigg\{\min\bigg\{B_{AIF}(\hat{r}(\T)),\\
&B(\gamma)-\l(\frac12+\hat{r}(\T)\r)J_0(\gamma),-b(0,\T)J_0(\gamma)\bigg\}+B(u_0^{(1)},\gamma),\\
&B_{BJ}(u_0^{(1)})\bigg\},
\end{align*}
which meets (\ref{def:limitfracqc}).

\textbf{Jump in $(0,1)$.} Assume that $S_u=\{x\}$, with $x\in(0,1)$. Let $(u_n)$ be a sequence converging to $u$ such that $\sup_n\H_{1,n}^{\ell,k_n,\T_n}(u_n)<\infty$. Then Proposition~\ref{prop:compact} implies that $u_n\to u$ in $L^1(0,1)$ with
\begin{align}
 u(t)=\begin{cases}
       \gamma t &\text{if } 0\leq t< x,\\
       (\ell-\gamma)+\gamma t &\text{if } x<t\leq1.
      \end{cases}\label{ucrackc}
\end{align}
Combining (\ref{frac:leftboundary}), (\ref{frac:rightboundary}) and the arguments of case (4) above, we can prove
\begin{equation}
 \liminf_{n\to\infty} \H_{1,n}^{\ell,k_n,\T_n}(u_n)\geq B(u_0^{(1)},\gamma)+B(u_1^{(1)},\gamma)-b(x,\T)J_0(\gamma)-J_0(\gamma),\label{eq:liminfcontinuous}\\
\end{equation}
which is the asserted estimate.\\

\textbf{\underline{Limsup inequality.}} As for the lower bound it is sufficient to consider a single jump either in $0$ or in $(0,1)$.

\textbf{Jump in $0$.} Corresponding to the cases (1)--(4), see (\ref{cases}), we construct sequences $(u_n^{(i)})$ with $u_n^{(i)}\to u$ for $i=1,...,4$, where $u$ is given by (\ref{ucrackxi}) such that
\begin{align}
(1)~\lim_n \H_{1,n}^{\ell,k_n,\T_n}(u_n^{(1)})\leq& B(u_1^{(1)},\gamma)+B(\gamma)+B_{b}(u_0^{(1)})+\frac12J_1(u_0^{(1)})-3J_0(\gamma),\label{eq:limsupcrackxi<<}\\
(2)~\lim_n \H_{1,n}^{\ell,k_n,\T_n}(u_n^{(2)})\leq& B(u_0^{(1)},\gamma)+B(u_1^{(1)},\gamma)+B_{AIF}(\hat{r}(\T)-1)-J_0(\gamma),\label{eq:limsupcrackxi<}\\
(3)~\lim_n \H_{1,n}^{\ell,k_n,\T_n}(u_n^{(3)})\leq& B(u_0^{(1)},\gamma)+B(u_1^{(1)},\gamma)+B(\gamma)-\l(\frac32 + \hat{r}(\T)\r)J_0(\gamma),\label{eq:limsupcrackxi=}\\
(4)~\lim_n \H_{1,n}^{\ell,k_n,\T_n}(u_n^{(4)})\leq& B(u_0^{(1)},\gamma)+B(u_1^{(1)},\gamma)-b(0,\T)J_0(\gamma)-J_0(\gamma).\label{eq:limsupcrackxi>}
\end{align}
To show these inequalities, we recall some definitions of sequences from \cite{NNN}. For a fixed $\eta>0$, we can find by definition (\ref{fracture:b}) of $B(\gamma)$, a function $\tilde{u}:\mathbb{N}\to\mathbb{R}$ and $\tilde{N}\in\mathbb{N}$ such that $\tilde{u}^0=0,\tilde{u}^{i+1}-\tilde{u}^i=\gamma$ if $i\geq\tilde{N}$ and
\begin{equation}\label{ineq:supfreeboundary}
\begin{split}
 &\frac{1}{2}J_1(\tilde{u}^1-\tilde{u}^0)+\sum_{i\geq0}\bigg\{J_2\left(\frac{\tilde{u}^{i+2}-\tilde{u}^i}{2}\right)+\frac{1}{2}J_1(\tilde{u}^{i+2}-\tilde{u}^{i+1})\\
 &\qquad +\frac{1}{2} J_1(\tilde{u}^{i+1}-\tilde{u}^i)-J_0(\gamma)\bigg\} \leq B(\gamma)+\eta.\end{split}
\end{equation}
In order to recover the elastic boundary layers at $0$ and $1$, we use the same sequences as in the case $\ell\leq\gamma$, cf. Theorem~\ref{theorem:elastic}. Let $v:\mathbb{N}\to\mathbb{R}$ and $N_1\in\mathbb{N}$ with $v^0=0,v^1=u_0^{(1)},v^{i+1}-v^i=\gamma$ if $i\geq N_1$ be such that (\ref{ineq:supleftboundary}) is satisfied and $w:-\mathbb{N}\to\mathbb{R}$ and $N_2\in\mathbb{N}$ with $w^0=0,w^0-w^{-1}=u_1^{(1)},w^i-w^{i-1}=\gamma$ if $i\leq -N_2$, such that (\ref{ineq:suprightboundary}) is satisfied.

\textit{\underline{Case (1):}} We construct a sequence $(u_n)$ converging in $L^1(0,1)$ to $u$, given in (\ref{ucrackxi}), satisfying (\ref{eq:limsupcrackxi<<}).
For this, we can use the same recovery sequence which is constructed for a jump in $0$ in \cite[Theorem 4.2]{NNN}. Let $\eta>0$. By definition (\ref{fracture:bb}) of $B_b(\theta)$, there exist $\hat{w}:-\N\to\R$ and $\hat{k}_0\in\N$ such that $\hat{w}^{-\hat{k}_0-1}=0, \hat{w}^{-\hat{k}_0}=u_0^{(1)}$ and
\begin{equation}\label{ineq:supfracboundary}
\begin{split}
 & \frac12J_1(\hat{w}^0-\hat{w}^{-1})+\sum_{i=\hat{k}_0+1}^0\bigg\{J_2\l(\frac{\hat{w}^i-\hat{w}^{i-2}}{2}\r)+\frac12J_1(\hat{w}^i-\hat{w}^{i-1})\\
 &\qquad +\frac12J_1(\hat{w}^{i-1}-\hat{w}^{i-2})-J_0(\gamma)\bigg\}\leq B_b(u_0^{(1)})+\eta.
\end{split}
\end{equation}
The recovery sequence $(u_n)$, which is given in \cite[Theorem 4.2]{NNN}, is defined means of the sequences $\tilde u$, $\hat w$ and $w$, as
\[u_n^i=\begin{cases}
   \lambda_n\hat{w}^{i-\hat{k}_0-1}&\mbox{if $0\leq i\leq \hat{k}_0+1$,}\\
   \ell+\lambda_n(w^{k_n^2+1-n}+\tilde{u}^{i-(\hat{k}_0+2)}-\tilde{u}^{k_n^2+1-(\hat{k}_0+2)})&\mbox{if $\hat{k}_0+2\leq i\leq k_n^2+1$},\\
   \ell+\lambda_nw^{i-n}&\mbox{if $k_n^2+1\leq i \leq n.$}
  \end{cases}
\]
Since $k_n^2$ is such that $\lim_{n\to\infty}k_n^2=\lim_{n\to\infty}(n-k_n^2)=+\infty$ we have for $n$ large enough
$$ k_n^1-(\hat k_0 +2)>\tilde N \quad\mbox{and}\quad k_n^2-n+2\leq -N_2. $$
In the proof of \cite[Theorem 4.2]{NNN} it is shown that $\lim_{n\to\infty}u_n=u$ in $L^1(0,1)$ and, by using the above inequalities, we can argue as in \cite{NNN} to show
\[\lim_n\H_{1,n}^{\ell,k_n,\T_n}(u_n)\leq \frac12J_1(u_0^{(1)})+B_b(u_0^{(1)})+B(\gamma)+B(u_1^{(1)},\gamma)-3J_0(\gamma)+3\eta.\]
The thesis follows from the arbitrariness of $\eta>0$.  

\textit{\underline{Case (2):}} Now we construct a sequence $(u_n)$ which converges in $L^1(0,1)$ to $u$, given in (\ref{ucrackxi}), and satisfies (\ref{eq:limsupcrackxi<}).\\
Let $\hat{r}(\T)<\infty$. For fixed $\eta>0$ we can find, by definition (\ref{fracture:bif}) of $B_{IF}(n)$, a function $z:\mathbb{N}\to\mathbb{R}$ and $q\in\mathbb{N}$ such that $z^0=0$ and 
\begin{equation}\label{ineq:tildeBsup}
 \begin{split}
  & \frac{1}{2}J_1(z^1-z^0)+\sum_{i=0}^{q-1}\bigg\{J_2\left(\frac{z^{i+2}-z^i}{2}\right)+\frac{1}{2}J_1(z^{i+2}-z^{i+1})+\frac{1}{2} J_1(z^{i+1}-z^i)-J_0(\gamma)\bigg\}\\
 &+\l(\frac{1}{2}+\hat{r}(\T)-1\r)\left(J_{CB}(z^{q+1}-z^q)-J_0(\gamma)\right)\leq  B_{IF}(\hat{r}(\T)-1)+\eta. 
 \end{split}
\end{equation}
Further, we extend $z$ such that $z^{i+1}-z^i=z^{q+1}-z^q$ for all $i\geq q$. Set $h_n:=k^1_n-q-1$, then we have $\lambda_n h_n\to0$. Moreover, let $(k_n^0)$ be a sequence of integers such that $\lambda_nk_n^0\to 0$ as $n\to\infty$ and
\begin{eqnarray*}
 k_n^0\geq N_1+1,\quad\tilde{N}\leq h_n-k_n^0-2,\quad n-k_n^2-1\geq N_2.		
\end{eqnarray*}
We are now able to construct a sequence $(u_n)$ by means of the functions $z,v,w$ and $\tilde{u}$, which is similar to the recovery sequence for an internal jump in \cite[p.~807]{NNN}
\begin{equation*}
 u_n^i=\begin{cases}
        \lambda_n v^i &\text{if } 0\leq i\leq k_n^0,\\
	\lambda_n(v^{k_n^0}-\tilde{u}^{h_n-i}+\tilde{u}^{h_n-k_n^0})&\text{if }k_n^0\leq i\leq h_n,\\
	\ell+\lambda_n(w^{k_n^1+\hat{r}(\T)-n}+z^{i-(h_n+1)}-z^{q+\hat{r}(\T)})&\text{if }h_n+1\leq i\leq r(\T_n),\\
	\ell+\lambda_n w^{i-n} &\text{if }r(\T_n)\leq i\leq n.
       \end{cases}
\end{equation*}
By definition of $v$ and $w$ the sequence $(u_n)$ satisfies the boundary conditions (\ref{def:boundarycond}). We have 
\[u_n^{k^1_n+1}-u_n^{k^1_n}=\lambda_n\l(z^{k_n^1-h_n}-z^{k_n^1-h_n-1}\r)=\lambda_n\l(z^{q+1}-z^q\r)\]
and by the definition of $z$ and $u_n$ this implies $u_n^{i+1}-u_n^i=z^{q+1}-z^q$ for $k_n^1\leq i<r(\T_n)$. Moreover, we have $u_n^{i+1}-u_n^i=\lambda_n\gamma$ for $N_1\leq i<h_n-\tilde{N}$ and $r(\T_n)\leq i < n-N_2$ which implies $u_n\in\A_{\T_n}(0,1)$. Since we have $k_n^1=h_n+q+1$, $r(\T_n)-k_n^1=\hat{r}(\T)$ and $k_n^2>k_n^1+\hat{r}(\T)$ for $n$ large enough, we obtain
\begin{equation*}
\begin{split}
 u_n^{h_n+1}-u_n^{h_n}=&\ell+\lambda_n\left(w^{k_n^1+\hat{r}(\T)-n}+z^0-z^{q+\hat{r}(\T)}-v^{k_n^0}+\tilde{u}^0-\tilde{u}^{h_n-k_n^0}\right)\\
=&\ell+\lambda_n\big(w^{k_n^1+\hat{r}(\T)-n}-w^{-N_2}+w^{-N_2}-z^{q+\hat{r}(\T)}-v^{k_n^0}+v^{N_1}-v^{N_1}\\
 &-\tilde{u}^{h_n-k_n^0}+\tilde{u}^{\tilde{N}}-\tilde{u}^{\tilde{N}}\big)\\
=&\ell+\lambda_n\big(\gamma(k_n^1+\hat{r}(\T)-n+N_2-k_n^0+N_1\\
 &-h_n+k_n^0+\tilde{N})+w^{-N_2}-z^{q+\hat{r}(\T)}-v^{N_1}-\tilde{u}^{\tilde{N}}\big)\\
=&\ell-n\gamma\lambda_n+\lambda_n\big(\gamma(q+1+\hat{r}(\T)+N_2+N_1+\tilde{N})\\
&+w^{-N_2}-z^{q+\hat{r}(\T)}-v^{N_1}-\tilde{u}^{\tilde{N}}\big).
\end{split}
\end{equation*}
Hence, we have
\begin{equation}
 u_n^{h_n+1}-u_n^{h_n}\to \ell-\gamma,\label{lim:lg2xi}
\end{equation}
and $u_n\to u$ in $L^1(0,1)$. From (\ref{lim:lg2xi}) we have $\frac{u_n^{h_n+1}-u_n^{h_n}}{\lambda_n}\to+\infty$ as $n\to\infty$ and thus
\begin{equation}\label{limsup:case1sigma}
\begin{split}
 \sigma_n^{h_n}&=\frac12J_1(z^1-z^0)-J_0(\gamma)+r_1(n),\\
 \sigma_n^{h_n-1}&=\frac12J_1(\tilde{u}^1-\tilde{u}^0)-J_0(\gamma)+r_2(n),
\end{split}
\end{equation}
with $r_1(n),r_2(n)\to0$ as $n\to\infty$. To compute $\H_{1,n}^{\ell,k_n,\T_n}(u_n)$, it is useful to write (\ref{def:hnl1<}) as follows
 \begin{equation*}
  \begin{split}
 \H_{1,n}^{\ell,k_n,\T_n}(u_n)=&\frac12J_1\l(u_0^{(1)}\r)+\sum_{i=0}^{k_n^0-2}\sigma_n^i+\sigma_n^{k_n^0-1}+\sum_{i=k_n^0}^{h_n-2}\sigma_n^i+\sigma_n^{h_n-1}+\sigma_n^{h_n}+\sum_{i=h_n+1}^{k_n^1-1}\sigma_n^i\\
 &+\frac12\mu_n^{k_n^1}+\sum_{i=k_n^1+1}^{k_n^2-2}\mu_n^i+\frac12\mu_n^{k_n^2-1}+\sum_{i=k_n^2-1}^{n-2}\sigma_n^i+\frac12J_1\l(u_1^{(1)}\r)-J_0(\gamma).
 \end{split}
 \end{equation*}
As in \cite[(4.69)]{NNN} we obtain $\sigma_n^{k_n^0-1}=0$. Combining (\ref{ineq:supleftboundary}), (\ref{ineq:suprightboundary}), (\ref{ineq:supfreeboundary}), (\ref{ineq:tildeBsup}) and (\ref{limsup:case1sigma}) we get
\begin{equation*}
 \begin{split}
 &\H_{1,n}^{\ell,k_n,\T_n}(u_n)=\frac12J_1(v^1-v^0)+\sum_{i\geq0}\bigg\{J_2\left(\frac{v^{i+2}-v^i}{2}\right)+\frac{1}{2}J_1(v^{i+2}-v^{i+1})\\
 &~+\frac{1}{2} J_1(v^{i+1}-v^i)-J_0(\gamma)\bigg\}+\frac12J_1(\tilde{u}^1-\tilde{u}^0)+\sum_{i\geq0}\bigg\{J_2\left(\frac{\tilde{u}^{i+2}-\tilde{u}^i}{2}\right)\\
 &~+\frac{1}{2}J_1(\tilde{u}^{i+2}-\tilde{u}^{i+1})+\frac{1}{2} J_1(\tilde{u}^{i+1}-\tilde{u}^i)-J_0(\gamma)\bigg\}+\frac12J_1(z^1-z^0)\\
 &~+\sum_{i=0}^{q-1}\bigg\{J_2\left(\frac{z^{i+2}-z^i}{2}\right)+\frac{1}{2}J_1(z^{i+2}-z^{i+1})+\frac{1}{2} J_1(z^{i+1}-z^i)-J_0(\gamma)\bigg\}\\
 &~+\l(\hat{r}(\T)-\frac12\r)\l(J_{CB}\l(z^{q+1}-z^q\r)-J_0(\gamma)\r)+\frac{1}{2}J_1(w^0-w^{-1})\\
 &~+\sum_{i\leq0}\left\{J_2\left(\frac{w^i-w^{i-2}}{2}\right)+\frac{1}{2}J_1(w^i-w^{i-1})+\frac{1}{2} J_1(w^{i-1}-w^{i-2})-J_0(\gamma)\right\}\\
 &~+r_1(n)+r_2(n)-3J_0(\gamma)\\
 &\leq B(u_0^{(1)},\gamma)+B(u_1^{(1)},\gamma)+B(\gamma)+ B_{IF}(\hat{r}(\T)-1)-3J_0(\gamma)+4\eta+r_1(n)+r_2(n)
 \end{split}
\end{equation*}
which yields (\ref{eq:limsupcrackxi<}).\\
Let now $\hat{r}(\T)=+\infty$. By definition, we have $B_{IF}(+\infty)=B(\gamma)$ and thus $B_{AIF}(+\infty)=B_{IJ}$ and we can use the same recovery sequence as used in case of an internal jump in Theorem 4.2. in \cite[p.~807]{NNN}. 

\textit{\underline{Case (3):}} We have to prove that there exists a sequence $(u_n)$ converging in $L^1(0,1)$ to $u$, given in (\ref{ucrackxi}), satisfying (\ref{eq:limsupcrackxi=}).\\
Without loss of generality we can assume that $\hat{r}(\T)<+\infty$, otherwise the inequality is trivial. Recall that $k_n^1=t_n^{k_n^1}$ by (\ref{def:atoms}), and hence $r(\T_n)=t_n^{k_n^1+1}$. Let $(k_n^0)_n\subset\N$ be such that $\lambda_nk_n^0\to 0$ as $n\to\infty$ and $k_n^0\geq N_1+1$. We now construct a sequence $(u_n)$ by means of the functions $v,w$ and $\tilde{u}$:
\begin{equation*}
 u_n^i=\begin{cases}
        \lambda_nv^i&\text{if }0\leq i\leq k_n^0,\\
	\lambda_n(v^{k_n^0}-\tilde{u}^{k_n^1-i}+\tilde{u}^{k_n^1-k_n^0})&\text{if }k_n^0\leq i\leq k^1_n,\\
	\l(\ell+\lambda_nw^{r(\T_n)-n}\r)\frac{i-k_n^1}{r(\T_n)-k_n^1}+\lambda_n\l(v^{k_n^0}+\tilde{u}^{k_n^1-k_n^0}\r)\frac{r(\T_n)-i}{r(\T_n)-k_n^1}&\text{if }k_n^1\leq i\leq r(\T_n),\\
	\ell+\lambda_nw^{i-n}&\text{if }r(\T_n)\leq i \leq n.
       \end{cases}
\end{equation*}
By definition of the function $v$ and $w$ the sequence $u_n$ satisfies the boundary conditions (\ref{def:boundarycond}). We have $u_n^{i+1}-u_n^i=\lambda_n\gamma$ for $N_1\leq i\leq k_n^1-\tilde{N}$ and $r(\T_n)\leq i\leq n-N_2$ for $n$ large enough. Since $u_n$ is affine on $\lambda_n(k_n^1,r(\T_n))$ we have $u_n\in\A_{\T_n}(0,1)$. Moreover, 
\begin{equation*}
\begin{split}
 u_n^{t_n^{k^1_n+1}}-u_n^{k^1_n}=&\ell+\lambda_n(w^{r(\T_n)-n}-v^{k_n^0}-\tilde{u}^0-\tilde{u}^{k_n^1-k_n^0})\\
=&\ell+\lambda_n(w^{r(\T_n)-n}-w^{-N_2}+v^{N_1}-v^{k_n^0}+\tilde{u}^{\tilde{N}}\\
 &-\tilde{u}^{k^1_n-k_n^0}-v^{N_1}-\tilde{u}^{\tilde{N}}+w^{-N_2})\\
=&\ell+\lambda_n\gamma(r(\T_n)-n+N_2-k_n^0+N_1+\tilde{N}-(k^1_n-k_n^0))\\
 &-\lambda_n(v^{N_1}+\tilde{u}^{\tilde{N}}+w^{-N_2})\\
=&\ell-n\lambda_n\gamma+\lambda_n(\gamma(N_1+\tilde{N}+N_2+\hat{r}(\T))-v^{N_1}-\tilde{u}^{\tilde{N}}+w^{-N_2}),
\end{split}
\end{equation*}
where we used $\hat r (\T)=r(\T_n)-k_n^1$ for $n$ large enough. Hence, we can conclude
\begin{equation}
 u_n^{r(\T_n)}-u_n^{k^1_n}\to \ell-\gamma.\label{lim:lg1xi}
\end{equation}
Thus, we have that $u_n$ converges to $u$ in $L^1(0,1)$. By using $u_n\in\A_{\T_n}(0,1)$ and (\ref{lim:lg1xi}) we obtain
\begin{equation*}
 \frac{u_n^{i+1}-u_n^i}{\lambda_n}=\frac{u_n^{r(\T_n)}-u_n^{k_n^1}}{(r(\T_n)-k_n^1)\lambda_n}\to\infty
\end{equation*}
as $n\to\infty$ for $k_n^1\leq i\leq r(\T_n)-1$. Hence 
\begin{equation*}
\begin{split}
 \sigma_n^{k_n^1-1}=&\frac12 J_1\l(\tilde{u}^1-\tilde{u}^0\r)-J_0(\gamma)+r_1(n)\\
 \mu_n^{i}=&-J_0(\gamma)+r_2(n),\qquad\mbox{for $k_n^1\leq i< r(\T_n)$}
\end{split}
\end{equation*}
with $r_1(n),r_2(n)\to0$ as $n\to\infty$. This leads, by using $\lim_{n\to\infty}r(\T_n)-k_n^1=\hat r\l(\T\r)$, to the estimate
\[\lim_{n\to\infty}\l(\frac{1}{2}\mu_n^{k_n^1}+\sum_{i=k_n^1+1}^{r(\T_n)-1}\mu_n^i\r)= -\l(\hat{r}(\T)-\frac12\r) J_0(\gamma).\]
Now similar calculations as before lead, by using (\ref{ineq:supleftboundary}), (\ref{ineq:suprightboundary}) and (\ref{ineq:supfreeboundary}), to 
\[\limsup_n\H_{1,n}^{\ell,k_n,\T_n}(u_n)\leq B(u_0^{(1)},\gamma)+B(u_1^{(1)},\gamma)+B(\gamma)-\l(\frac{3}{2}+\hat{r}(\T)\r)J_0(\gamma)+3\eta\]
which proves (\ref{eq:limsupcrackxi=}) by the arbitrariness of $\eta>0$.

\textit{\underline{Case (4):}} Here, we prove that there exists a sequence $(u_n)$ converging in $L^1(0,1)$ to $u$, given by (\ref{ucrackxi}), which satisfies (\ref{eq:limsupcrackxi>}).\\
Without loss of generality we can assume $b(0,\T)<+\infty$. By the definition of $b(0,\T)$, we can find a sequence $(h_{n})$ such that 
\[\lim_{n\to\infty}(t_n^{h_n+1}-t_n^{h_n})=b(0,\T).\]
We construct now the sequence $(u_n)$ by means of the functions $v$ and $w$:
 \[u_n^i=\begin{cases}
          \lambda_n v^i &\text{if } 0\leq i\leq t_n^{h_n},\\
 	  \frac{t_n^{h_n+1}-i}{t_n^{h_n+1}-t_n^{h_n}}\lambda_nv^{t_n^{h_n}}+\frac{i-t_n^{h_n}}{t_n^{h_n+1}-t_n^{h_n}}(\ell+\lambda_nw^{t_n^{h_n+1}-n})&\text{if } t_n^{h_n}\leq i\leq t_n^{h_n+1},\\
 	  \ell+\lambda_nw^{i-n} &\text{if } t_n^{h_n+1}\leq i \leq n.
         \end{cases}
 \]
 This sequence satisfies the boundary conditions (\ref{def:boundarycond}) and $u_n^{i+1}-u_n^i=\lambda_n\gamma$ for $N_1\leq i\leq t_n^{h_n}$ and for $t_n^{h_n+1}\leq i\leq n-N_2$ and we have
\begin{equation*}\label{lim:lg1c}
\begin{split}
 u_n^{t_n^{h_n+1}}-u_n^{t_n^{h_n}}=&\ell+\lambda_n(w^{t_n^{h_n+1}-n}-v^{t_n^{h_n}})\\
=&\ell+\lambda_n(w^{t_n^{h_n+1}-n}-w^{-N_2}+w^{-N_2}-v^{t_n^{h_n}}+v^{N_1}-v^{N_1})\\
=&\ell+\lambda_n(\gamma(t_n^{h_n+1}-t_n^{h_n}-n+N_2+N_1)+w^{-N_2}-v^{N_1})\to\ell-\gamma.
\end{split}
\end{equation*}
Thus, $u_n\to u$ in $L^1(0,1)$. Furthermore, we obtain for $t_n^{h_n}\leq i\leq t_n^{h_n+1}-1$,% by using (\ref{lim:lg1c})
\[\mu_n^i=J_{CB}\l(\frac{u_n^{t_n^{h_n+1}}-u_n^{t_n^{h_n}}}{\lambda_n(t_n^{h_n+1}-t_n^{h_n})}\r)-J_0(\gamma)\to-J_0(\gamma)\]
as $n\to\infty$. This implies 
\[\sum_{i=t_n^{h_n}}^{t_n^{h_n+1}-1}\mu_n^i=-b(0,\T)J_0(\gamma),\]
and together with (\ref{ineq:supleftboundary}) and (\ref{ineq:suprightboundary}) the desired inequality (\ref{eq:limsupcrackxi>}) follows.

\textbf{Jump in $(0,1)$} We have to prove that there exists a sequence $(u_n)$ converging in $L^1(0,1)$ to $u$, given in (\ref{ucrackc}), satisfying 
$$\lim_n \H_{1,n}^{\ell,k_n,\T_n}(u_n)\leq B(u_0^{(1)},\gamma)+B(u_1^{(1)},\gamma)-b(x,\T)J_0(\gamma)-J_0(\gamma).$$
This can be shown analogously to case (4) for a jump in $0$, by using sequence $(h_n)\subset\N$ with $t_n^{h_n}, t_n^{h_n+1}\in\T_n$ for all $n\in\N$ such that
$$ \lim_{n\to\infty}(t_n^{h_n+1}-t_n^{h_n})=b(x,\T).$$ 
\end{proof}

%
%%Chapter 5
%

\section{Minimum Problems}

According to Theorem~\ref{th:fracturennn} and Theorem~\ref{theorem:fracture}, the functionals $\H_{1,n}^{\ell,k_n,\T_n}$ and $H_{1,n}^\ell$ do not have the same $\Gamma$-limit for $\ell>\gamma$, while they coincide in the case $\ell\leq\gamma$. In order to analyze the validity of the QC-approximation also for $\ell>\gamma$, we study the minimum of $\H_1^{\ell,\T}$ in dependence of the choice of representative atoms described by $\T$. We give sufficient conditions on $\T$ such that $\min_uH_1^\ell(u)=\min_u\H_1^{\ell,\T}(u)$. Moreover, we give examples in which the minimal energies and minimizers of $H_1^\ell$ and $\H_1^{\ell,\T}$ do not coincide. To this end, certain relations between different boundary layer and jump energies are needed, which we provide in several lemmas at the beginning of this section. Some of these relations are proven under additional though quite general assumptions on the potentials $J_1$ and $J_2$. In Proposition~\ref{proplj}, we show that all these assumptions are satisfied for the classical Lennard-Jones and Morse potentials, see (\ref{def:LJ}) and (\ref{def:morse}). First, let us recall some estimates for the boundary layer energies from \cite{NNN}.
\begin{lemma}\cite[Lemma 5.1]{NNN}\label{lemma:NNN51}
Let [LJ1]--[LJ4] be satisfied. Then
\begin{enumerate}
 \item[(1)] $\frac{1}{2}J_1(\delta_1)\leq B(\gamma)\leq\frac{1}{2}J_1(\gamma)$;
 \item[(2)] $B(\theta,\gamma)\geq\frac{1}{2}J_1(\theta)$ for all $\theta>0$;
 \item[(3)] $B_b(\theta)\geq\frac{1}{2}J_1(\delta_1)$ for all $\theta>0$;
 \item[(4)] $B_b(\delta_1)=\frac{1}{2}J_1(\delta_1)$.
\end{enumerate}
\end{lemma}

In this chapter, we also need a similar estimate for $B_{IF}(m)$ as for $B(\gamma)$ and an upper bound for $B_b(\theta)$.
\begin{lemma}\label{lemma:B}
 Let [LJ1]--[LJ4] be satisfied. Then 
\begin{equation}\label{ineq:bifm}
 \frac{1}{2}J_1(\delta_1)\leq B_{IF}(m)\leq\frac{1}{2}J_1(\gamma)
\end{equation}
for every $m\in\mathbb{N}\cup\{+\infty\}$ and $B_b(\theta)\leq\frac12J_1(\theta)$, where $B_{IF}(m)$ and $B_b(\theta)$ are defined in \eqref{fracture:bif} and \eqref{fracture:bb}.
\end{lemma}
\begin{proof}
 We can argue as in \cite[Lemma 5.1 (1)]{NNN}. The sum in the definition of $B_{IF}(m)$, see \eqref{fracture:bif}, is non-negative since $\gamma$ is the minimum point of $J_0$ and we have
$$B_{IF}(m)\geq \min \frac12J_1(z)=\frac12J_1(\delta_1).$$
To show the upper bound, we can use the function $u:\N\to\R$ with $u^i=i\gamma$ as a competitor for $B_{IF}(m)$ for every $m\in\N$ and deduce the upper bound. The estimate for $B_b(\theta)$ follows by choosing $k=0$ in definition (\ref{fracture:bb}).  
\end{proof}

 To compare $\min_uH_1^\ell(u)$ and $\min_u \H_1^{\ell,\T}(u)$, we need to estimate $B_{IFJ}(n,k,\theta)$, defined in (\ref{def:bifjxint}). This will be done, under additional assumptions on $J_1,J_2$, in the following lemmas.
\begin{lemma}\label{lemma:bifj} 
Let $J_1,J_2$ be such that [LJ1]--[LJ4] are satisfied and $J_1(\gamma)$, $J_2(\gamma)$, $J_2(\delta_1)<0$. Define the quantity
\begin{equation}
 \tilde{B}_{IFJ}(n,k):=\min\left\{B_{AIF}(n),B(\gamma)-\l(\frac12+n\r)J_0(\gamma),-kJ_0(\gamma)\right\},
\end{equation}
where $B_{AIF}$ is as in (\ref{def:baif}). Then 
 \begin{itemize}
  \item[(i)] $\tilde{B}_{IFJ}(n,1)=-J_0(\gamma)$ for all $n\in\N\cup\{+\infty\}$, $n\geq1$,
  \item[(ii)] $\tilde{B}_{IFJ}(1,k)=B(\gamma)-\frac32J_0(\gamma)$ for all $k\in\N\cup\{+\infty\},k\geq2$,
  \item[(iii)] $\tilde{B}_{IFJ}(n,k)=B_{AIF}(n)$ for all $n,k\in\N\cup\{+\infty\}$ with $n\geq2,k\geq2$.
 \end{itemize}
\end{lemma}
\begin{proof}
(i) From $J_2(\delta_1)<0$, we deduce $J_0(\gamma)\leq J_0(\delta_1)\leq J_1(\delta_1)+J_2(\delta_1)<J_1(\delta_1)$. Hence, we obtain by $B(\gamma),B_{IF}(n)\geq\frac12J_1(\delta_1)$, see Lemma \ref{lemma:B} (1) and \eqref{ineq:bifm}, and the definition of $B_{AIF}(n)$, see (\ref{def:baif}), that
\begin{align*}
 B_{AIF}(n)\geq& J_1(\delta_1)-2J_0(\gamma)>-J_0(\gamma),\\
 B(\gamma)-\l(\frac32+n\r)J_0(\gamma)\geq&B(\gamma)-\frac32J_0(\gamma)\geq \frac12J_1(\delta_1)-\frac32J_0(\gamma)>-J_0(\gamma).
\end{align*}
(ii) From $B_{IF}(m)\geq\frac12J_1(\delta_1)$, $0>J_1(\delta_1)>J_0(\gamma)$ and $B(\gamma)\leq\frac12J_1(\gamma)<0$, $J_0(\gamma)<J_1(\gamma)$, we deduce %
\begin{align*}
 B_{AIF}(1)\geq&\frac12J_1(\delta_1)+B(\gamma)-2J_0(\gamma)>B(\gamma)-\frac32J_0(\gamma),\\
 -kJ_0(\gamma)\geq&-2J_0(\gamma)>\frac12J_1(\gamma)-\frac32J_0(\gamma)\geq B(\gamma)-\frac32J_0(\gamma).
\end{align*}
(iii) Again by $B_{IF}(m),B(\gamma)\leq \frac12J_1(\gamma)<0$ and $J_0(\gamma)<0$, we conclude
\begin{align*}
 B_{AIF}(n)\leq&\frac12J_1(\gamma)+B(\gamma)-2J_0(\gamma)<B(\gamma)-\frac52J_0(\gamma)\\
 B_{AIF}(n)\leq&J_1(\gamma)-2J_0(\gamma)<-kJ_0(\gamma),
\end{align*}
which proves the statement.
\end{proof}

In order to compute the value of $B_{IFJ}(n,k,\theta)$, see (\ref{def:bifjxint}), we provide an estimate for $B_{AIF}(n)$.% {\color{magenta}defined in \eqref{def:baif}}.
\begin{lemma}\label{theorem:bif}
Let $J_1,J_2$ satisfy assumptions [LJ1]--[LJ4] and additionally 
\begin{equation}
 R(t):=J_2\l(\frac{\gamma+t}{2}\r)+\frac12\l(J_1(\gamma)+J_1(t) \r) - J_0(\gamma) - \frac32 \l( J_{CB}(t)-J_0(\gamma) \r)\leq0\label{ineq:chapter5a}
\end{equation}
for all $t\in\operatorname{dom}J_1$. Then $B_{IF}(m)= B(\gamma)$ for any $m\geq1$ and $B_{AIF}(n)=B_{IJ}$ for $n\geq2$, where $B_{IF}(m)$, $B(\gamma)$, $B_{AIF}(n)$ and $B_{IJ}$ are defined in \eqref{fracture:bif}, \eqref{fracture:b}, \eqref{def:baif} and \eqref{eq:fractureIJ}. 
\end{lemma}

\begin{proof}
Let us first show that $B_{IF}(m)\leq B(\gamma)$. For every $\eta>0$ there exists, by the definition of $B(\gamma)$, in (\ref{fracture:b}), a function  $\tilde{u}:\N\to\R$ and $\tilde{N}\in\N$ such that $\tilde u^0 = 0$, $\tilde u^{i+1} - \tilde u^i=\gamma$ if $i\geq \tilde N$, satisfying (\ref{ineq:supfreeboundary}). The function $\tilde{u}$ is also a competitor for the minimum problem for $B_{IF}(m)$, see \eqref{fracture:bif}. Hence, we have for some $k>\tilde{N}+1$ 
\begin{align*}
 B_{IF}(m)\leq& \frac{1}{2}J_1(\tilde{u}^1-\tilde{u}^0)+\sum_{i=0}^{k-1}\bigg\{J_2\left(\frac{\tilde{u}^{i+2}-\tilde{u}^i}{2}\right)+\frac{1}{2}J_1(\tilde{u}^{i+2}-\tilde{u}^{i+1})\\
 &+\frac{1}{2} J_1(\tilde{u}^{i+1}-\tilde{u}^i)-J_0(\gamma)\bigg\}+\frac{2m+1}{2}\left(J_{CB}(\tilde{u}^{k+1}-\tilde{u}^k)-J_0(\gamma)\right)\\
\leq& B(\gamma)+\eta%
\end{align*}
and the assertion follows by the arbitrariness of $\eta>0$. \\
Let us now show $B_{IF}(m)\geq B(\gamma)$ for $m\geq1$. The definition of $B_{IF}(m)$, see (\ref{fracture:bif}), implies $B_{IF}(m)\geq B_{IF}(1)$ for all $m\geq1$. Let $\eta>0$. By the definition of $B_{IF}(1)$ in (\ref{fracture:bif}) there exists $u:\N\to\R$ with $u^0=0$, and $k\in\N$ such that
\begin{align*}
 & \frac{1}{2}J_1(u^1-u^0)+\sum_{i=0}^{k-1}\bigg\{J_2\left(\frac{u^{i+2}-u^i}{2}\right)+\frac{1}{2}J_1(u^{i+2}-u^{i+1})+\frac{1}{2} J_1(u^{i+1}-u^i)\\
 &-J_0(\gamma)\bigg\}+\frac{3}{2}\left(J_{CB}(u^{k+1}-u^k)-J_0(\gamma)\right)\leq  B_{IF}(1)+\eta.
\end{align*}
If we extend $u$ such that $u^{i+1}-u^i=\gamma$ for $i\geq k+1$, $u$ becomes a competitor for $B(\gamma)$, see \eqref{fracture:b}. Thus
\begin{align*}
 B(\gamma)\leq& \frac{1}{2}J_1(u^1-u^0)+\sum_{i\geq0}\bigg\{J_2\left(\frac{u^{i+2}-u^i}{2}\right)+\frac{1}{2}J_1(u^{i+2}-u^{i+1})+\frac{1}{2} J_1(u^{i+1}-u^i)\\
 &-J_0(\gamma)\bigg\}\leq B_{IF}(1)+\eta+R(u^{k+1}-u^k).
\end{align*}
By assumption (\ref{ineq:chapter5a}), we have $R(u^{k+1}-u^k)\leq 0$. Hence, by the arbitrariness of $\eta>0$, we have $B_{IF}(m)\geq B_{IF}(1)\geq B(\gamma)$ for all $m\geq1$. \\ 
Altogether, we have $B_{IF}(m)=B(\gamma)$ for $m\geq1$. Hence, we have by the definition of $B_{AIF}(n)$ and $B_{IJ}$, see (\ref{def:baif}) and (\ref{eq:fractureIJ}), that $B_{AIF}(n)=B_{IJ}$ for $n\geq2$. 
\end{proof}

Before we state our main result of this section, we show some estimates for the boundary layer energies in $H_1^\ell$, see (\ref{def:limitfrac}).
\begin{lemma}\label{lem:bbjbij}
 Let $J_1,J_2$ satisfy [LJ1]--[LJ4]. Then 
\begin{equation}\label{ineq:bbj<bij}
 B(\theta,\gamma)\leq B_{BJ}(\theta)\leq B(\theta,\gamma)+B_{IJ}\qquad\forall\theta>0,
\end{equation}
and $B_{IJ}>0$, where $B(\theta,\gamma)$, $B_{BJ}(\theta)$ and $B_{IJ}$ are defined in \eqref{eq:elasticboundary}, \eqref{eq:fractureBJ} and \eqref{eq:fractureIJ}. If, for $\theta>0$, there exists a constant $\eta_\theta>0$ such that $\frac12J_1(\gamma)+J_2\left(\frac{t+\gamma}{2} \right)\leq0$ for all $t\in\R$ with $J_1(t)<J_1(\theta)+2\eta_\theta$, it holds $B(\theta,\gamma)< B_{BJ}(\theta)$.
\end{lemma}

\begin{proof}
Let $\ell>\gamma$ and $u_0^{(1)}=u_1^{(1)}=\theta$. The inequalities of (\ref{ineq:bbj<bij}) and $B_{IJ}>0$ follow from the lower semicontinuity of $H_1^\ell$ given in (\ref{def:limitfrac}). Indeed, by the properties of the $\Gamma$-limit, we deduce that $H_1^\ell$ is lower semicontinuous with respect to the strong $L^1(0,1)$--topology, see e.g.\ \cite[Proposition 1.28]{Bbeg}. Let $u\in SBV_c^\ell(0,1)$ be such that $S_u=\{0\}$. Furthermore, define $(u_n),(v_n)\subset SBV_c^\ell(0,1)$ such that $S_{u_n}=\left\{\frac{1}{n}\right\}$ and $S_{v_n}\subset\{0,1\}$ with $[v_n](1)=\frac{\ell-\gamma}{n}$. Note that $u$, $u_n$ and $v_n$ with $n\in\N$, $n\geq1$ are uniquely defined. Since, $(u_n)$ and $(v_n)$ converge strongly in $L^1(0,1)$ to $u$, we deduce from the lower semicontinuity of $H_1^\ell$:
\begin{equation*}
 \begin{split}
   B(\theta,\gamma)+B_{BJ}(\theta)-J_0(\gamma)=&H_1^\ell(u)\leq \liminf_{n\to\infty}H_1^\ell(u_n)\leq2B(\theta,\gamma)+B_{IJ}-J_0(\gamma),\\
 B(\theta,\gamma)+B_{BJ}(\theta)-J_0(\gamma)=&H_1^\ell(u)\leq\liminf_{n\to\infty}H_1^\ell(v_n)\leq 2B_{BJ}(\theta)-J_0(\gamma).
 \end{split}
\end{equation*}
Hence, (\ref{ineq:bbj<bij}) is proven. Let us show $B_{IJ}>0$. Similarly to the upper bound in the zeroth-order $\Gamma$-limit (Theorem~\ref{theorem:zero}), we can construct a sequence $(w_n)\subset SBV^\ell_c(0,1)$ such that $\#S_{w_n}=n$ and $w_n\to u$ in $L^1(0,1)$ with $u(x)=\ell x$. If we assume on the contrary that $B_{IJ}\leq 0$, we had $\sup_n H_1^\ell(w_n)\leq C$ but $H_1^\ell(u)=+\infty$ since $u\notin SBV_c^\ell(0,1)$ for $\ell>\gamma$, which was a contradiction to the lower semicontinuity of $H_1^\ell$. Thus $B_{IJ}>0$. \\
Next, we prove $B(\theta,\gamma)<B_{BJ}(\theta)$ under the additional assumption. Let $\eta>0$ be such that $\eta<\eta_\theta$ and $\frac12B_{IJ}-\eta>0$. We show $B_{BJ}(\theta)-(\frac12B_{IJ}-\eta)\geq B(\theta,\gamma)$, which clearly proves $B(\theta,\gamma)<B_{BJ}(\theta)$. By the definition of $B_b(\theta)$, see (\ref{fracture:bb}), there exists $k\in\N$ and $(v^i)_{i=0}^{k+1}\in\R^{k+2}$ such that $v^{k+1}=0$ and $v^k=-\theta$ with
\begin{align*}
 B_b(\theta)+\eta\geq&\frac12J_1(v^1-v^0)+\sum_{i=0}^{k-1}\bigg\{J_2\left(\frac{v^{i+2}-v^i}{2}\right)+\frac{1}{2}J_1(v^{i+2}-v^{i+1})\\
&+\frac{1}{2} J_1(v^{i+1}-v^i)-J_0(\gamma)\bigg\}.
\end{align*}
By the upper bound $B_b(\theta)\leq \frac12 J_1(\theta)$, see Lemma~\ref{lemma:B}, and the fact that the terms in the above sum are non-negative, we deduce $J_1(v^1-v^0)\leq J_1(\theta)+2\eta$. Let us define the sequence $u=(u^i)_{i=0}^\infty$ by $u^i=-v^{k+1-i}$ for $i\in\{0,..., k+1\}$ and $u^{i+1}-u^i=\gamma$ for $i\geq k+1$. Since the sequence $u$ is a competitor for the minimum problem which defines $B(\theta,\gamma)$, see (\ref{eq:elasticboundary}), we have
\begin{align*}
 &B(\theta,\gamma)\\
&\leq\frac12J_1(u^1-u^0)+\sum_{i\geq0}\bigg\{J_2\left(\frac{u^{i+2}-u^i}{2}\right)+\frac12J_1(u^{i+2}-u^{i+1})\\
&~+\frac12J_1(u^{i+1}-u^i)-J_0(\gamma)\bigg\}\\
&=\frac12J_1(\theta)+\sum_{i=0}^{k-1}\left\{J_2\left(\frac{v^{i+2}-v^i}{2}\right)+\frac{1}{2}J_1(v^{i+2}-v^{i+1})+\frac{1}{2} J_1(v^{i+1}-v^i)-J_0(\gamma)\right\}\\
&~+J_2\left(\frac{\gamma+v^1-v^0}{2}\right)+\frac12J_1(v^1-v^0)+\frac12J_1(\gamma)-J_0(\gamma)\\
&\leq\frac12J_1(\theta)+B_b(\theta)+\eta-J_0(\gamma)= B_{BJ}(\theta)+\eta-(B(\gamma)-J_0(\gamma))\\
&=B_{BJ}(\theta)-\left(\frac12B_{IJ}-\eta\right),%<B_{BJ}(\theta)
\end{align*}
where we used $\frac12J_1(\gamma)+J_2\left(\frac{v^1-v^0+\gamma}{2}\right)\leq0$. 
\end{proof}

As a direct consequence of Lemma~\ref{lem:bbjbij}, we have the following result about the minimizers and minimal energies of $H_1^\ell$, which extends in some sense the results of \cite[Theorem 5.1]{NNN}. We prove that there exists no choice for $u_0^{(1)},u_1^{(1)}>0$ such that an internal jump has strictly less energy than a jump at the boundary. However, note that for special values of $u_0^{(1)},u_1^{(1)}>0$ the energies can be the same.
\begin{proposition}
  Suppose that hypotheses [LJ1]--[LJ4] hold. Let $\ell>\gamma$. For any $u_0^{(1)},u_1^{(1)}>0$ it holds 
\begin{equation}\label{minh1l}
\min_u H_1^\ell(u)=\min\left\{B_{BJ}\left(u_0^{(1)}\right)+B\left(u_1^{(1)},\gamma\right),B_{BJ}\left(u_1^{(1)}\right)+B\left(u_0^{(1)},\gamma\right)\right\}-J_0(\gamma).
\end{equation}
\end{proposition}

\begin{proof}
 From $B_{BJ}(\theta)\leq B(\theta,\gamma)+B_{IJ}$ for all $\theta>0$, see Lemma~\ref{lem:bbjbij} and the formula for $H_1^\ell$ in (\ref{def:limitfrac}), it follows that no internal jump can has strictly less energy than a jump at the boundary. Hence, 
$$\min\left\{H_1^\ell(u):u\in SBV_c^\ell(0,1)\right\}=\min \left\{H_1^\ell(u):u\in SBV_c^\ell(0,1),S_u\subset\{0,1\}\right\},$$
which proves, using $B(\theta,\gamma)\leq B_{BJ}(\theta)$ (see (\ref{ineq:bbj<bij})), the assertion (\ref{minh1l}), cf.~(\ref{def:limitfrac}).
\end{proof}

Combining the previous results, we are able to give sufficient conditions on the representative atoms $\T=(\T_n)$ in order to ensure $\min_uH_1^\ell(u)=\min_u\H_1^{\ell,\T}(u)$. In plain terms, it is enough to make sure that the representative atoms $\T_n$ are such that $k_n^1+1, k_n^2-1\notin \T_n$ and for all $i,j\in\{k_n^1+1,....,k_n^2-1\}\cap \T_n$ it holds $|i-j|\geq2$. 

\begin{theorem}\label{5:cor}
 Let $u_0^{(1)},u_1^{(1)}>0$ and $\ell>\gamma$. Let $J_1,J_2$ satisfy [LJ1]--[LJ4], $J_1(\gamma),J_2(\gamma),J_2(\delta_1)<0$ and (\ref{ineq:chapter5a}). If $\T=(\T_n)$ satisfies (\ref{def:atoms}) and $b(x,\T)$, $\hat{l}(\T)$, $\hat{r}(\T)\geq 2$, see (\ref{def:B(x,T)}), (\ref{def:rl}), for all $x\in(0,1)$, then $\H_1^{\ell,\T}$ defined in (\ref{def:limitfracqc}) reads
\begin{equation}\label{def:limitfracljcoarse}
\H_1^{\ell,\T}(u)=H_1^\ell(u)-\sum_{x:x\in S_u\cap(0,1)}\left(b(x,\T)J_0(\gamma)+B_{IJ}\right)
\end{equation}
for $u\in SBV^\ell_c(0,1)$, and $+\infty$ else on $L^1(0,1)$. Moreover, for given $u_0^{(1)},u_1^{(1)}>0$
\begin{equation}
 \min_u\H_1^{\ell,\T}(u)=\min_u H_1^\ell(u).\label{min=}
\end{equation}
For $u\in \operatorname{argmin} \H_1^{\ell,\T}$, the jump set satisfies $S_u\subset\{0,1\}$. If furthermore $J_1$ and $J_2$ satisfy all assumptions of Lemma~\ref{lem:bbjbij}, it holds $\#S_u=1$.
\end{theorem}
\begin{proof}
Let us first prove (\ref{def:limitfracljcoarse}). By the definition of $H_1^\ell$ and $\H_1^{\ell,\T}$, see (\ref{def:limitfrac}), (\ref{def:limitfracqc}), we have to show $ B_{IFJ}(\hat r(\T),b(0,\T),u_0^{(1)})=B_{BJ}(u_0^{(1)})$ and $B_{IFJ}(\hat{l}(\T),b(1,\T),u_1^{(1)})=B_{BJ}(u_1^{(1)})$. By Lemma~\ref{theorem:bif}, we have $B_{AIF}(n)=B_{IJ}$, for $n\geq 2$. Hence, we have for $B_{IFJ}(n,k,\theta)$, defined in (\ref{def:bifjxint}), with $n,k\geq2$ and $\theta>0$ by Lemma~\ref{lemma:bifj} (iii)  and inequality (\ref{ineq:bbj<bij}) that
\begin{equation*}
  B_{IFJ}(n,k,\theta)=\min\left\{B_{AIF}(n)+B(\theta,\gamma),B_{BJ}(\theta)\right\}=B_{BJ}(\theta).
\end{equation*}
Hence, by  $b(x,\T),\hat{l}(\T),\hat{r}(\T)\geq 2$, for all $x\in(0,1)$ the assertion (\ref{def:limitfracljcoarse}) is proven.\\
From $J_0(\gamma)<0$, Lemma~\ref{lemma:bifj} (iii), Lemma~\ref{theorem:bif} and Lemma~\ref{lem:bbjbij}, we deduce that
\begin{equation}\label{b(x,T)>0}
 -b(x,\T)J_0(\gamma)\geq -2J_0(\gamma)>\tilde B_{IFJ}(2,2)=B_{AIF}(2)=B_{IJ}>0
\end{equation}
for all $x\in(0,1)$. Combining (\ref{b(x,T)>0}) with (\ref{ineq:bbj<bij}), we obtain that $B_{BJ}(\theta)<B(\theta,\gamma)-2J_0(\gamma)$ for all $\theta>0$. Hence, the jump set $S_u$ of minimizers $u$ of $\H_1^{\ell,\T}$ satisfies $S_u\subset\{0,1\}$ and by (\ref{ineq:bbj<bij})--(\ref{def:limitfracljcoarse})
\begin{align*}
 \min_u \H_1^{\ell,\T}(u)=&\min\left\{B_{BJ}\left(u_0^{(1)}\right)+B\left(u_1^{(1)},\gamma\right),B_{BJ}\left(u_1^{(1)}\right)+B\left(u_0^{(1)},\gamma\right)\right\}-J_0(\gamma)\\
=&\min_u H_1^{\ell}(u).
\end{align*}
If $J_1$ and $J_2$ are such that $B(\theta,\gamma)<B_{BJ}(\theta)$ for all $\theta>0$, see Lemma~\ref{lem:bbjbij}, we obtain from the above equation that every minimizer $u$ of $\H_1^{\ell,\T}$ satisfies $\#S_u=1$.
\end{proof}

In the next theorem which is based on the previous $\Gamma$-convergence statements, we deduce a convergence result for the difference between the minimal energies of the fully atomistic model and the quasicontinuum model.  
\begin{theorem}\label{th:limmin}
 Let $u_0^{(1)},u_1^{(1)}>0$, $\ell>0$ and let $k_n^1,k_n^2$ satisfy \eqref{ass:kn}. Let $J_1$, $J_2$ and $(\T_n)$ satisfy the assumptions of Theorem~\ref{theorem:elastic} and, if $\ell>\gamma$, also the additional assumptions of Theorem~\ref{theorem:fracture} and Theorem~\ref{5:cor} such that \eqref{min=} is valid. Then it holds
\begin{equation}\label{lim:discretemin}
 \inf_u H_n^\ell(u)-\inf_u \H_n^{\ell,k_n,\T_n}(u)=o(\lambda_n),
\end{equation}
as $n\to\infty$.
\end{theorem}
\begin{proof}
 Let us first note that the functionals $H_n^\ell$, $\H_n^{\ell,k_n,\T_n}$ are equi-coercive in $L^1(0,1)$, which follows by the compactness argument in the proof of Theorem~\ref{theorem:zero}. Moreover, by Proposition \ref{prop:compact} the functionals $H_{1,n}^\ell$, $\H_{1,n}^{\ell,k_n,\T_n}$ are equi-coercive. In the case $0<\ell\leq\gamma$, Theorem \ref{theorem:elastic} ensures that $H_n^\ell$ and $\H_n^{\ell,k_n,\T_n}$ are $\Gamma$-equivalent at order $\lambda_n$, see \cite[Definition 4.2]{BT}, and \eqref{lim:discretemin} follows from \cite[Theorem 4.4]{BT}. Similarly, if $\gamma<\ell$, we deduce from Theorem~\ref{theorem:zero} and Theorem~\ref{theorem:fracture}  
\begin{align*}
\inf_u \H_n^{\ell,k_n,\T_n}(u)=&\inf_u H^\ell(u)+\lambda_n\inf_u \hat H_1^{\ell,\T}(u)+o(\lambda_n),%\\
% =&\inf_u H^\ell(u)+\lambda_n\inf_u H_1^\ell (u)+o(\lambda_n)=\inf_u H_n^\ell(u)+o(\lambda_n), 
\end{align*}
see \cite[Theorem 1.47]{Bbeg}. Further, by \eqref{min=} and Theorem~\ref{th:fracturennn}, we obtain
\begin{align*}
\inf_u \H_n^{\ell,k_n,\T_n}(u)=\inf_u H^\ell(u)+\lambda_n\inf_u H_1^\ell (u)+o(\lambda_n)=\inf_u H_n^\ell(u)+o(\lambda_n). 
\end{align*}
\end{proof}

In the next proposition, we show that the sufficient conditions of Theorem \ref{5:cor} are sharp. Therefore, we show for a particular choice of $u_0^{(1)},u_1^{(1)}>0$ that if the representative atoms are not chosen as in the above theorem, neither the minimal energy nor the minimizer of $\H_1^{\ell,\T}$ coincide with the ones of $H_1^\ell$.
\begin{proposition}\label{lemma:fracture}
Let $\ell>\gamma$, $u_0^{(1)}=\delta_1$ and $u_1^{(1)}=\gamma$. Let $J_1,J_2$ satisfy [LJ1]--[LJ4]. Then it holds for $H_1^\ell$ 
\begin{equation}\label{min:H1l}
 \min_u H_1^\ell(u)=B_{BJ}(\delta_1)+B(\gamma,\gamma)-J_0(\gamma),
\end{equation}
and the unique minimizer $u$ satisfies $S_u=\{0\}$. Let $J_1,J_2$ satisfy the assumptions of Theorem~\ref{5:cor} and $J_2(\gamma)>2J_2\left(\frac{\delta_1+\gamma}{2}\right)$. Then the following assertions hold true:
\begin{itemize}
 \item[(a)] Let $\T^1=(\T_n^1)$ be such that there exists $z\in[0,1]$ with $b(z,\T^1)=1$. Then $\min_u \H_{1}^{\ell,\T^1}=B(\delta_1,\gamma)+B(\gamma,\gamma)-2J_0(\gamma)<\min_uH_1^\ell$ and the jump appears indifferently in $z\in[0,1]$ with $b(z,\T^1)=1$.
\item[(b)] Let $\T^2=(\T_n^2)$ be such that $\hat{l}(\T^2)=1$ and $\hat{r}(\T^2),b(z,\T^2)\geq2$ for all $z\in[0,1]$. Then $\min_u \H_1^{\ell,\T^2}=B(\delta_1,\gamma)+B(\gamma,\gamma)+B(\gamma)-\frac32J_0(\gamma)<\min_uH_1^\ell$ and the jump appears in $1$.
\end{itemize}
\end{proposition}
\begin{proof}
Let us first prove the part regarding the energy $H_1^\ell$. It is shown in \cite[Theorem 5.1]{NNN} that $B_{BJ}(\delta_1)<B(\delta_1,\gamma)+B_{IJ}$ and $B_{BJ}(\gamma)=B(\gamma,\gamma)+B_{IJ}$. This implies
\begin{equation}\label{ineq:deltagamma}
 B_{BJ}(\delta_1)+B(\gamma,\gamma)<B(\delta_1,\gamma)+B(\gamma,\gamma)+B_{IJ}=B(\delta_1,\gamma)+B_{BJ}(\gamma),
\end{equation}
which proves (\ref{min:H1l}) and that the unique minimizer $u$ of $H_1^\ell$ satisfies $S_u=\{0\}$. Let us now show the assertions concerning the minimal energies of $\H_1^{\ell,\T}$. We test the minimum problem for $B(\delta_1,\gamma)$, see (\ref{eq:elasticboundary}), with $v:\N\to\R$ such that $v^{i+1}-v^i=\gamma$ for all $i\geq1$. By using $J_2(\gamma)>2J_2\left(\frac{\delta_1+\gamma}{2}\right)$ and $J_0(\gamma)=J_1(\gamma)+J_2(\gamma)$, we obtain
\begin{equation}\label{elasticupbd}
B(\delta_1,\gamma)\leq J_1(\delta_1)+\frac12J_1(\gamma)+J_2\l(\frac{\delta_1+\gamma}{2}\r)-J_0(\gamma)< J_1(\delta_1)-\frac12J_0(\gamma). 
\end{equation}
From (\ref{def:bifjxint}) and Lemma~\ref{lemma:bifj}, we deduce $B_{IFJ}(n,k,\theta)\geq\min\{-J_0(\gamma)+B(\theta,\gamma),B_{BJ}(\theta)\}$.\\
(a) Combining the above considerations with (\ref{def:limitfracqc}) it is enough to show that $B(\delta_1,\gamma)-J_0(\gamma)<B_{BJ}(\delta_1)$. This follows by using (\ref{elasticupbd}), Lemma~\ref{lemma:NNN51} (1), (4) and $J_0(\gamma)<J_1(\delta_1)$:
$$
B(\delta_1,\gamma)-J_0(\gamma)<J_1(\delta_1)-\frac32J_0(\gamma)\leq\frac{1}{2}J_1(\delta_1)+B_b(\delta_1)+B(\gamma)-2J_0(\gamma)=B_{BJ}(\delta_1).%\\
$$
(b) From (\ref{def:limitfracqc}), Theorem~\ref{5:cor} and $\hat r(\T^2),b(z,\T^2)\geq2$ for all $z\in[0,1]$, we deduce $\H_1^{\ell,\T^2}(u)\geq \min H_1^\ell$ for $u\in SBV_c^\ell(0,1)$ with $S_u\cap[0,1)\neq\emptyset$. Let us compute the energy for a jump at $1$: For $k\geq2$, we have by Lemma~\ref{lemma:bifj} (ii) that $\tilde B_{IFJ}(1,k)=B(\gamma)-\frac32J_0(\gamma)$. As in Lemma~\ref{lemma:bifj} (ii), we have, by using $B(\gamma)\geq\frac12J_1(\delta_1)>\frac12J_0(\gamma)$ if $J_2(\gamma)<0$, that $B_{IJ}\geq B(\gamma)-\frac32J_0(\gamma)$. Hence, by applying $B_{BJ}(\gamma)=B(\gamma,\gamma)+B_{IJ}$ and the definition of $B_{IFJ}(n,k,\theta)$, see (\ref{def:bifjxint}), we deduce  
$$B_{IFJ}(1,k,\gamma)=\min\left\{B(\gamma)-\frac{3}{2}J_0(\gamma),B_{IJ}\right\}+B(\gamma,\gamma)=B(\gamma)-\frac{3}{2}J_0(\gamma)+B(\gamma,\gamma).$$
Thus, we deduce from $\hat l(\T^2)=1$ and $b(1,\T^2)=2$ that $B_{IFJ}(\hat l(\T^2),b(1,\T^2),\gamma)=B(\gamma)-\frac32J_0(\gamma)+B(\gamma,\gamma)$. Hence, by the definition of $\H_1^{\ell,\T}$, see \eqref{def:limitfracqc}, and by \eqref{min:H1l} it remains to show that $B(\delta_1,\gamma)+B(\gamma)-\frac32J_0(\gamma)<B_{BJ}(\delta_1)$, which follows by using (\ref{elasticupbd}) and Lemma~\ref{lemma:NNN51} (1), (4)
\begin{align*}
 B(\delta_1,\gamma)+B(\gamma)-\frac32J_0(\gamma)<& J_1(\delta_1)+B(\gamma)-2J_0(\gamma)\\
=&\frac12J_1(\delta_1)+B_b(\delta_1)+B(\gamma)-2J_0(\gamma)=B_{BJ}(\delta_1).
\end{align*}
\end{proof}

We conclude this section by showing that all additional assumptions on $J_1,J_2$ in this chapter are satisfied by the classical Lennard-Jones potentials and Morse potentials, defined in (\ref{def:LJ}) and (\ref{def:morse}) respectively.
\begin{proposition}\label{proplj}
 Let $J_1,J_2$ be as in (\ref{def:LJ}) or (\ref{def:morse}) respectively. Then $J_1$ and $J_2$ satisfy $J_1(\gamma)$, $J_2(\gamma), J_2(\delta_1)<0$, $J_2(\gamma)>2J_2\left(\frac{\delta_1+\gamma}{2}\right)$ and inequality (\ref{ineq:chapter5a}) holds on $\operatorname{dom} J_1$. Furthermore, there exists for all $\theta>0$ a constant $\eta_\theta>0$ such that $J_2\left(\frac{t+\gamma}{2}\right)<0$ for $t\in\operatorname{dom} J_1$ such that $J_1(t)< J_1(\theta)+2\eta_\theta$.
\end{proposition}
\begin{proof}
 Let $J_1,J_2$ satisfy (\ref{def:LJ}), i.e., there exist $k_1,k_2>0$ such that $J_1(z)=\frac{k_1}{z^{12}}-\frac{k_2}{z^6}$ and $J_2(z)=J_1(2z)$. Straightforward calculations lead to 
\begin{equation}\label{eq:lj}
 \delta_1=\l(\frac{2k_1}{k_2}\r)^{1/6},\quad\gamma=\l(\frac{1+2^{-12}}{1+2^{-6}}\r)^{1/6}\delta_1,\quad z_0=\left(\frac{k_1}{k_2}\right)^{1/6}=\left(\frac{1}{2}\right)^{1/6}\delta_1,
\end{equation}
where $\delta_1$ is the unique minimizer of $J_1$, $\gamma$ the unique minimizer of $J_0$ (and $J_{CB}$) and $z_0$ is the unique zero of $J_1$ with $J_1<0$ on $(z_0,+\infty)$. Note that $z_0<\gamma<\delta_1$. Moreover, we have that $J_1$ is strictly decreasing on $(0,\delta_1)$ and strictly increasing on $(\delta_1,+\infty)$. 
%A simple calculation yield $J_1(z)<0$ for $z>\left(\frac{k_1}{k_2}\right)^{1/6}:=z_0$. 
From $\gamma>z_0$, we deduce $J_1(\gamma)<0$ and thus $J_2\left(\frac{\gamma+t}{2}\right)=J_1(\gamma+t)<0$ on $\{t:t>0\}=\operatorname{dom}J_1$. Since $\gamma<2\gamma<2\delta_1$, we have $J_2(\gamma),J_2(\delta_1)<0$. Moreover, by $\delta_1/2<\gamma<\delta_1$ and the definition of $J_2$, it is sufficient to show $J_2(\gamma)>2J_2(\delta_1)$ to obtain $J_2(\gamma)>2J_2\left(\frac{\delta_1+\gamma}{2}\right)$:
\begin{align*}
 J_2(\gamma)-2J_2(\delta_1)=&\frac{k_1}{2^{12}\delta_1^{12}}\left(\frac{(1+2^{-6})^2}{(1+2^{-12})^2}-2\right)-\frac{k_2}{2^6\delta_1^6}\left(\frac{1+2^{-6}}{1+2^{-12}}-2\right)\\
=&\frac{k_2^2}{4k_12^{12}}\left(\frac{(1+2^{-6})^2}{(1+2^{-12})^2}-2-2^7\left(\frac{1+2^{-6}}{1+2^{-12}}-2\right)\right)>0.
\end{align*}
Let us now show inequality (\ref{ineq:chapter5a}). Since $J_0(\gamma)=J_{CB}(\gamma)=J_1(\gamma)+J_2(\gamma)$ and $J_0'(\gamma)=J_{CB}'(\gamma)=0$ one directly has $R(\gamma)=0$ and $R'(\gamma)=0$. Consider the function $J_1+2J_2$ given by
\[J_1(z)+2J_2(z)=\frac{k_1}{z^{12}}-\frac{k_2}{z^6}+\frac{k_1}{2^{11}z^{12}}-\frac{k_2}{2^5z^6}=\frac{k_1(1+2^{-11})}{z^{12}}-\frac{k_2(1+2^{-5})}{z^6}.\]
This is again a Lennard-Jones potential and there exists a constant $z_c>0$ such that $J_1''(z)+2J_2''(z)>0$ for all $z\in(0,z_c)$. To compute $z_c$ we set the second derivative of $J_1+2J_2$ equal to zero:
\[0=\frac{156k_1(1+2^{-11})}{z_c^{14}}-\frac{42k_2(1+2^{-5})}{z_c^8},\quad z_c>0\quad\Leftrightarrow\quad z_c=\delta_1\l(\frac{13}{7}\frac{1+2^{-11}}{1+2^{-5}}\r)^{1/6}.\]
From an analogous calculation we obtain that $J_{CB}''(z)>0$ for $z\in(0,z_*)$ with \\ $z_*=\delta_1\l(\frac{13}{7}\frac{1+2^{-12}}{1+2^{-6}}\r)^{1/6}>z_c$.
Now we estimate $R$ on $[z_c,+\infty)$. Since $z_c>\delta_1>\gamma$, we have $\frac12J_1-\frac32J_{CB}=-\frac12J_2-J_{CB}$ is decreasing on $(z_c,+\infty)$. Since $J_2\l(\frac{t+\gamma}{2}\r)=J_1(t+\gamma)<0$ for $t\geq0$, we have 
$$R(t)\leq-\frac12J_2(z_c)-J_{CB}(z_c)+\frac{1}{2}(J_1(\gamma)+J_0(\gamma))\approx-0.0469\frac{k_2^2}{k_1}<0, $$
for $t\geq z_c$. We now show that $R'(t)\geq0$ for $t\leq\gamma$ and $R'(t)\leq0$ for $\gamma\leq t\leq z_c$, which proves the statement. For $0<t\leq\gamma<z_c<z_*$, we have
\begin{align*}
 R'(t)=&\frac12J_2'\l(\frac{t+\gamma}{2}\r)+\frac12J_1'(t)-\frac32J_{CB}'(t)=\frac12\left(J_2'\l(\frac{t+\gamma}{2}\r)-J_2'(t)\right)-J_{CB}'(t)\\%\frac12\int_{t}^{\frac{t+\gamma}{2}}J_2''(z)dz-\int_\gamma^tJ_{CB}''(z)dz\\
=&\frac12\int_{t}^{\frac{t+\gamma}{2}}J_2''(z)dz+\int_t^\gamma J_{CB}''(z)dz\geq\frac12\int_{t}^{\frac{t+\gamma}{2}}J_2''(z)+J_{CB}''(z)dz>0.
\end{align*}
Analogously we get for $\gamma\leq t\leq z_c$
$$R'(t) = -\frac12\int_{\frac{t+\gamma}{2}}^{t}J_2''(z)dz-\int_\gamma^t J_{CB}''(z)dz\leq-\frac12\int_{\frac{t+\gamma}{2}}^{t}J_2''(z)+J_{CB}''(z)dz<0. $$
Hence, Lennard-Jones potentials satisfy all the properties asserted.

Let now $J_1$ and $J_2$ be Morse potentials as in (\ref{def:morse}), i.e., there exist $k_1,k_2,\delta_1>0$ such that $J_1(z)=k_1\left(1-e^{-k_2(z-\delta_1)}\right)^2-k_1$ and $J_2(z)=J_1(2z)$. In this case, we do not have such an explicit expression for $\gamma$ as in the Lennard-Jones case and therefore derive bounds on $\gamma$. Since $J_1'(z)<0$ iff $z<\delta_1$ and $J_1'(z)>0$ iff $z>\delta_1$, we deduce from $0=J_{CB}'(\gamma)=J_1'(\gamma)+2J_1'(2\gamma)$ that $\delta_1/2<\gamma<\delta_1$. A straightforward calculation yields $J_1(z)< 0$ iff $z>\frac{k_2\delta_1-\ln(2)}{k_2}=:z_0$. In order to prove $J_1(\gamma)<0$, we show $J'_{CB}(z_0)<0$, which implies $z_0<\gamma$. Indeed, we have
$$J_{CB}'(z_0)=-4k_1k_2 \left(16e^{-2k_2\delta_1} - 4 e^{-k_2 \delta_1} + 1\right)<0.$$
% $$J_{CB}'(z_0)=-4k_1k_2 \left(16e^{-2k_2\delta_1} - 4 e^{-k_2 \delta_1} + 1\right)=-4k_1k_2\left((1-2e^{-k_2\delta_1})^2+12e^{-2k_2\delta_1}\right)<0.$$
As in the Lennard-Jones case, we deduce from $J_1(\gamma)<0$, $\gamma<\delta_1$ and the definition of $J_2$ that $J_2(\gamma),J_2(\delta_1)<0$ and $J_2\left(\frac{\gamma+t}{2}\right)<0$ for all $t>0$. Define for $\theta>0$ the constant $\eta_\theta:=\frac{1}{2}(J_1(0)-J_1(\theta))>0$, then we deduce $J_2\left(\frac{t+\gamma}{2}\right)<0$ for $t\in\{t:J_1(t)<J_1(\theta)+2\eta_\theta\}\subset\{t:t>0\}$.\\ 
Let us show $J_2(\gamma)-2J_2\left(\frac{\delta_1+\gamma}{2}\right)=J_1(2\gamma)-2J_1(\delta_1+\gamma)>0$. From $\{\gamma\}=\operatorname{argmin}J_{CB}$, we deduce 
\begin{equation*}%\label{morsegamma1}
\begin{split}
 0=J_{CB}'(\gamma)=&-k_1k_2\left(-2e^{k_2\delta_1}(e^{-k_2\gamma}+2e^{-2k_2\gamma})+e^{2k_2\delta_1}(2e^{-2k_2\gamma}+4e^{-4k_2\gamma})\right)\\
=&2k_1k_2e^{k_2\delta_1}e^{-4k_2\gamma}\left(e^{3k_2\gamma}+2e^{2k_2\gamma}-e^{k_2\delta_1}(2+e^{2k_2\gamma})\right)\\
=&2k_1k_2q_{\delta_1}q_\gamma^{-4}\left(q_\gamma^3+2q_\gamma^2-q_{\delta_1}(2+q_\gamma^2)\right)
\end{split}
\end{equation*}
with $q_\gamma:=e^{k_2\gamma}>1$ and $q_{\delta_1}:=e^{k_2\delta_1}>1$. This yields $q_{\delta_1}=\frac{q_\gamma^3+2q_\gamma^2}{2+q_\gamma^2}$ and allows us to show%Using (\ref{morsegamma1}), we obtain
\begin{align}\label{morsegamma2}
J_2(\gamma)-2J_2\left(\frac{\delta_1+\gamma}{2}\right)=&k_1\left(-2e^{-k_2(2\gamma-\delta_1)}+e^{-2k_2(2\gamma-\delta_1)}+4e^{-k_2\gamma}-2e^{-2k_2\gamma}\right)\notag\\
=&k_1e^{-4k_2\gamma}\left(-2e^{k_2\delta_1}e^{2k_2\gamma}+e^{2k_2\delta_1}+4e^{3k_2\gamma}-2e^{2k_2\gamma}\right)\notag\\
=&k_1q_\gamma^{-4}\left(4q_\gamma^3-2(1+q_{\delta_1})q_\gamma^2+q_{\delta_1}^2\right)\notag\\
=&\frac{k_1}{q_\gamma^{2}(q_\gamma^2+2)^2}\left(2q_\gamma^5-5q_\gamma^4+16 q_\gamma^3-12q_\gamma^2+16 q_\gamma-8\right)%\\
%=&\frac{k_1}{q_\gamma^{2}(q_\gamma^2+2)^2}\left(q_\gamma^3\left(\sqrt 2 q_\gamma -\frac{5}{2\sqrt 2}\right)^2+\frac{103}{8}q_\gamma^3-12q_\gamma^2+16q_\gamma-8\right).
% =&\frac{k_1}{q_\gamma^{2}(q_\gamma^2+2)^2}\left(q_\gamma^3\left(\sqrt 2 q_\gamma -\frac{5}{2\sqrt 2}\right)^2+\frac18q_\gamma^3+12q_\gamma^2(q_\gamma-1)+16q_\gamma-8\right).
% =&\frac{k_1}{q_\gamma^{2}(q_\gamma^2+2)^2}\left(q_\gamma^3\left(\sqrt 2 q_\gamma -\frac{5}{2\sqrt 2}\right)^2+\frac18q_\gamma^3+12q_\gamma^2(q_\gamma-1)+8q_\gamma+8(q_\gamma-1)\right).%>0,%\notag\\
\end{align}
The assertion follows since (\ref{morsegamma2}) is positive for $q_\gamma>1$.\\
It is left to show that $R=R(t)\leq0$ for all $t\in\R$. We prove the inequality in a different way than in the Lennard-Jones case. We have $\lim_{t\to+\infty}R(t)=\frac12J_1(\gamma)+\frac12J_0(\gamma)<0$ and by using $J_1(t+\gamma)<J_1(2t)$ for $t<0$ we obtain that
$$\lim_{t\to-\infty}R(t)\leq\lim_{t\to-\infty}\left(-J_1(t)-\frac12J_2(t)+\frac12J_1(\gamma)+\frac12J_0(\gamma)\right)=-\infty.$$
Moreover, by the definition of $R=R(t)$ and $\gamma$, we have that $R(\gamma)=R'(\gamma)=0$. To show that $R(t)\leq0$ it is sufficient to show that $R$ has no critical point except $\gamma$. Indeed, if $R(t)>0$ for some $t\in\R$, then in order to satisfy the conditions at infinity there has to exist a maximum point $\hat t$ with $R(\hat t)>0$ and $R'(\hat t)=0$. By the definition of $J_1$, $J_2$ and $R=R(t)$, we have
\begin{align*}
R'(t)=&J_1'(t+\gamma)-J_1'(t)-3J_1'(2t)\\ 
=&2k_1k_2e^{k_2\delta_1}\bigg(e^{-k_2(t+\gamma)}(1-e^{-k_2(t+\gamma-\delta_1)})-e^{-k_2t}(1-e^{-k_2(t-\delta_1)})\\
&~-3e^{-2k_2t}(1-e^{-k_2(2t-\delta_1)})\bigg)\\
=&2k_1k_2e^{k_2\delta_1}e^{-4k_2t}\left((e^{-k_2\gamma}-1)e^{3k_2t}+(e^{k_2\delta_1}(1-e^{-2k_2\gamma})-3)e^{2k_2t}+3e^{k_2\delta_1}\right)\\
=&2k_1k_2e^{k_2\delta_1}q_t^{-4}\left((e^{-k_2\gamma}-1)q_t^3+(e^{k_2\delta_1}(1-e^{-2k_2\gamma})-3)q_t^2+3e^{k_2\delta_1}\right)\\
=&2k_1k_2e^{k_2\delta_1}q_t^{-4}f(q_t)
\end{align*}
with $q_t=e^{k_2t}$. From $R'(\gamma)=0$ it follows $f(q_\gamma)=0$. Let us show that $q_\gamma$ is the unique zero of $f$. We have $f(0)=3e^{k_2\delta_1}>0$ and from $k_2,\gamma>0$, we deduce $e^{-k_2\gamma}-1<0$ and thus $\lim_{q\to\infty}f(q)=-\infty$. This implies that if $f$ had a second zero, it would have a local minimum and a local maximum. But
$$f'(q)=q\left(3(e^{-k_2\gamma}-1)q+2(e^{k_2\delta_1}(1-e^{-2k_2\gamma})-3)\right)$$
and thus $f$ has at most one local extremum in $(0,+\infty)$. Hence, $q_\gamma$ is the unique zero of $f$ and $\gamma$ the unique zero of $R'(t)$. 
\end{proof}

%\section*{Acknowledgment}  This work is supported by a grant of the Deutsche Forschungsgemeinschaft (DFG) SCHL 1706/2-1.

\end{document}